%% file: main.tex
\pdfoutput=1
\documentclass[12pt,reqno,a4paper]{amsart}

\usepackage[utf8]{inputenc}

\usepackage[doi=true,isbn=false,style=alphabetic,sorting=anyt,backend=biber,maxnames=12,maxalphanames=6,giveninits=true]{biblatex}
\AtEveryBibitem{\clearlist{language}}
\bibliography{main}

\usepackage[breaklinks,colorlinks,plainpages,hypertexnames=false,plainpages=false]{hyperref}
\hypersetup{urlcolor=blue, citecolor=blue, linkcolor=blue}

\usepackage{amsmath}
\usepackage{amsthm}
\usepackage{amssymb}
\usepackage{amsfonts}
\usepackage{mathtools}
\usepackage{hhline}
\usepackage{stmaryrd}
\usepackage{enumitem}
\usepackage{centernot}
\usepackage[textwidth=25mm,textsize=footnotesize]{todonotes}
\usepackage{mathrsfs}
\usepackage[version=4]{mhchem}
\usepackage{dsfont}
\usepackage{comment}
\usepackage{listings}
\usepackage[capitalise, noabbrev]{cleveref} %for \cref
\usepackage[noend]{algorithmic} % for the algorithmic environment

\usepackage{mathdots}
\usepackage{cleveref} % for \cref
\usepackage{nicematrix} % for NiceMatrix
\usepackage{mathalpha} % for various math fonts
\DeclareMathAlphabet{\dutchcal}{U}{dutchcal}{m}{n}
\newcommand{\dcF}{\mathcal{F}}
\newcommand{\dcf}{\dutchcal{f}}
\newcommand{\dcg}{\dutchcal{g}}
\newcommand{\dch}{\dutchcal{h}}
\newcommand{\dck}{\dutchcal{k}}
\newcommand{\dcq}{\dutchcal{q}}
\newcommand{\dcI}{\mathcal{I}}
\usepackage{xspace}
\usepackage{xcolor}

\usepackage{array}
\newcolumntype{L}[1]{>{\raggedright\let\newline\\\arraybackslash\hspace{0pt}}m{#1}}
\newcolumntype{C}[1]{>{\centering\let\newline\\\arraybackslash\hspace{0pt}}m{#1}}
\newcolumntype{R}[1]{>{\raggedleft\let\newline\\\arraybackslash\hspace{0pt}}m{#1}}

\usepackage{tikz-cd}
\usepackage{tikz}
\usetikzlibrary{arrows}
\usetikzlibrary{decorations.pathreplacing}
\usetikzlibrary{matrix}
\usetikzlibrary{calc}
\usetikzlibrary{shapes}
\usetikzlibrary{patterns}
\usetikzlibrary{fit,backgrounds,scopes,plotmarks}

\tikzset{ % for left and right brackets in tikz
    ncbar angle/.initial=90,
    ncbar/.style={
        to path=(\tikztostart)
        -- ($(\tikztostart)!#1!\pgfkeysvalueof{/tikz/ncbar angle}:(\tikztotarget)$)
        -- ($(\tikztotarget)!($(\tikztostart)!#1!\pgfkeysvalueof{/tikz/ncbar angle}:(\tikztotarget)$)!\pgfkeysvalueof{/tikz/ncbar angle}:(\tikztostart)$)
        -- (\tikztotarget)
    },
    ncbar/.default=0.5cm,
}
\tikzset{square left brace/.style={ncbar=1.5mm}}
\tikzset{square right brace/.style={ncbar=-1.5mm}}
\tikzset{round left paren/.style={ncbar=0.5cm,out=120,in=-120}}
\tikzset{round right paren/.style={ncbar=0.5cm,out=60,in=-60}}
\newtheoremstyle{theoremstyle}
{10pt}      %  Space above
{5pt}       %  Space below
{\itshape}  %  Body font
{}          %  Indent amount (empty = no indent, \parindent = para indent)
{\bfseries} %  Thm head font
{.}         %  Punctuation after thm head
{ }      %  Space after thm head: " " = normal interword space;
{}          %  Thm head spec (can be left empty, meaning `normal')

\newtheoremstyle{algorithmstyle}
{10pt}      %  Space above
{5pt}       %  Space below
{}  %  Body font
{}          %  Indent amount (empty = no indent, \parindent = para indent)
{\bfseries} %  Thm head font
{.}         %  Punctuation after thm head
{ }      %  Space after thm head: " " = normal interword space;
{}          %  Thm head spec (can be left empty, meaning `normal')

\newtheoremstyle{examplestyle}
{10pt}      %  Space abovedestroys shield
{5pt}       %  Space below
{}          %  Body font
{}          %  Indent amount (empty = no indent, \parindent = para indent)
{\bfseries} %  Thm head font
{.}         %  Punctuation after thm head
{ }      %  Space after thm head: " " = normal interword space;
{}          %  Thm head spec (can be left empty, meaning `normal')

\makeatletter
\newtheorem*{rep@theorem}{\rep@title}
\newcommand{\newreptheorem}[2]{%
\newenvironment{rep#1}[1]{%
 \def\rep@title{#2 \ref{##1}}%
 \begin{rep@theorem}}%
 {\end{rep@theorem}}}
\makeatother

\makeatletter % subalign = substack + align
\newcommand{\subalign}[1]{%
  \vcenter{%
    \Let@ \restore@math@cr \default@tag
    \baselineskip\fontdimen10 \scriptfont\tw@
    \advance\baselineskip\fontdimen12 \scriptfont\tw@
    \lineskip\thr@@\fontdimen8 \scriptfont\thr@@
    \lineskiplimit\lineskip
    \ialign{\hfil$\m@th\scriptstyle##$&$\m@th\scriptstyle{}##$\hfil\crcr
      #1\crcr
    }%
  }%
}
\makeatother

\newreptheorem{theorem}{Theorem}
\newreptheorem{corollary}{Corollary}
\newreptheorem{proposition}{Proposition}

\theoremstyle{theoremstyle}
\newtheorem{theorem}{Theorem}[section]
\newtheorem{lemma}[theorem]{Lemma}
\newtheorem{proposition}[theorem]{Proposition}
\newtheorem{corollary}[theorem]{Corollary}
\theoremstyle{examplestyle}
\newtheorem{example}[theorem]{Example}
\newtheorem{definition}[theorem]{Definition}

\newtheorem*{plan*}{Plan}
\newtheorem{remark}[theorem]{Remark}
\newtheorem{convention}[theorem]{Convention}

\theoremstyle{algorithmstyle}
\newtheorem{algorithm}[theorem]{Algorithm}

\numberwithin{equation}{section}

\definecolor{darkorange}{rgb}{1.0, 0.55, 0.0}
\definecolor{darkblue}{rgb}{0.0, 0.0, 0.55}
\definecolor{darkgreen}{rgb}{0.0, 0.2, 0.13}
\definecolor{darkred}{rgb}{0.75, 0.0, 0.0}

\newcommand{\CC}{\mathbb{C}}
\newcommand{\RR}{\mathbb{R}}
\newcommand{\QQ}{\mathbb{Q}}
\newcommand{\ZZ}{\mathbb{Z}}

\newcommand{\suchthat}{\;\ifnum\currentgrouptype=16 \middle\fi|\;}
\newcommand{\bigmid}{\left.\vphantom{\Big\{} \suchthat \vphantom{\Big\}}\right.}
\newcommand{\bigslant}[2]{{\raisebox{.2em}{$#1$}\left/\raisebox{-.2em}{$#2$}\right.}}

\newcommand{\lin}{{\mathrm{lin}}}
\newcommand{\nlin}{{\mathrm{nlin}}}

\newcommand{\bin}{{\mathrm{bin}}}

\newcommand{\CCt}{\mathbb{C}\{\!\{t\}\!\}}

\DeclareMathOperator{\initial}{in}

\DeclareMathOperator{\mult}{mult}
\DeclareMathOperator{\MV}{MV}

\DeclareMathOperator{\trop}{trop}
\DeclareMathOperator{\Trop}{Trop}
\DeclareMathOperator{\val}{val}

\newcommand\restr[2]{{\left.\kern-\nulldelimiterspace #1 \right|_{#2}}}

\makeatletter
\newcommand{\oset}[3][0ex]{%
  \mathrel{\mathop{#3}\limits^{
    \vbox to#1{\kern-2\ex@
    \hbox{$\scriptstyle#2$}\vss}}}}
\makeatother

\makeatletter
\newcommand{\uset}[3][0ex]{%
  \mathrel{\mathop{#3}\limits_{
    \vbox to#1{\kern-7\ex@
    \hbox{$\scriptstyle#2$}\vss}}}}
\makeatother
\makeatletter
\newcommand{\customlabel}[2]{%
   \protected@write \@auxout {}{\string \newlabel {#1}{{#2}{\thepage}{#2}{#1}{}} }%
   \hypertarget{#1}{#2}%
}
\makeatother
\usepackage[indent,skip=.2\baselineskip]{parskip}

\begin{document}

\title[A tropical method for solving
parametrized polynomial systems]{A tropical method for solving \\ parametrized polynomial systems}

\author{Paul Alexander Helminck}
\address{Mathematical Institute, Tohoku University, Japan.}
\email{paul.helminck.a6@tohoku.ac.jp}
\urladdr{https://paulhelminck.wordpress.com}

\author{Oskar Henriksson}
\address{Department of Mathematical Sciences, University of Copenhagen, Denmark.}
\email{oskar.henriksson@math.ku.dk}
\urladdr{https://oskarhenriksson.se}

\author{Yue Ren}
\address{Department of Mathematical Sciences, Durham University, United Kingdom.}
\email{yue.ren2@durham.ac.uk}
\urladdr{https://yueren.de}

% \subjclass[2020]{14T90,65H14,13P15}

% \date{\today}

% \keywords{tropical geometry, numerical algebraic geometry, polynomial system solving.}

\begin{abstract}
    We give a framework for constructing generically optimal homotopies for parametrized polynomial systems from tropical data.  Here, generically optimal means that the number of paths tracked is equal to the generic number of solutions. We focus on two types of parametrized systems -- vertically parametrized and horizontally parametrized systems -- and discuss techniques for computing the tropical data efficiently. We end the paper with several case studies, where we analyze systems arising from chemical reaction networks, coupled oscillators, and rigid graphs.
\end{abstract}

\maketitle

\input{introduction}

\input{background}

\input{tropicalHomotopies}

\input{verticalModifications}

\input{horizontalModifications}

\input{caseStudies}

\input{conclusion}

\renewcommand*{\bibfont}{\small}
\printbibliography
\end{document}

%% file: introduction.tex
\section{Introduction}
Solving systems of polynomial equations is a fundamental task throughout applied mathematics; for instance, polynomials govern the motion of robots \cite{SommeseWampler05}, the phase of coupled oscillators \cite{CMMN19}, and the concentrations of species in  biological systems \cite{Dickenstein16}.
A staple for solving polynomial systems numerically over the complex numbers is \emph{homotopy continuation} \cite{BBCHLS2023}, which traces the solutions of an easy-to-solve start system to the desired solutions of the target system along a path in the space of polynomial systems, commonly called a \emph{homotopy}.

Homotopy continuation is known for being able to compute a single solution to a polynomial system in average polynomial time, thereby answering Smale's 17th problem in the positive \cite{BP11,BC11,Lairez17}.  However, constructing homotopies for computing all solutions to a polynomial system remains a major challenge.

Ideally, a homotopy should be both fast to construct and \emph{optimal} in the sense that the number of paths equals the number of solutions of the target system.  However, constructing optimal homotopies requires a priori knowledge of the number of solutions of the target system, which is difficult to obtain efficiently.  This means that in practice, there is a tradeoff between the speed of construction and minimizing the number of superfluous paths.

For instance, computing the upper bound on the number of solutions given by B\'ezout's theorem is fast, as it only requires multiplying degrees. The (normalized) mixed volume bound given by Bernstein's theorem \cite{Bernstein1975}, on the other hand, will often be lower than the B\'ezout bound for sparse systems, but is substantially harder to compute.
The actual number of solutions can, in principle, be computed through a Gr\"obner basis computation, but in this case, homotopy continuation becomes irrelevant since there are more effective solvers available if given a Gr\"obner basis, such as eigenvalue solvers (see \cite[Chapter~2]{CoxLittleOShea05} for an introduction, and  \cite[Section~2.1]{Cox2020} for an overview of recent progress).

The mixed volume is known to be a sharp bound provided that the coefficients of the system are generic, but for many parametrized systems that arise in applications, relations among the coefficients lead to fewer than mixed volume many solutions (see, e.g., \cite{Gross2016wnt} and \cite{BBPMZ23}).  This has prompted a wealth of works on finding upper bounds on the number of solutions under weaker genericity assumptions, involving techniques such as tropical geometry \cite{HelminckRen22,HoltRen2023}, Khovanskii bases and Newton--Okounkov bodies \cite{KavehKhovanskii12,OW22,BBPMZ23}, and toric geometry and facial subsystems \cite{BKKSS21,BreidingSottileWoodcock22,LindbergMoninRose23}.

Given a bound on the number of solutions, it is a problem in its own right to construct a homotopy (or a collection of homotopies) that trace precisely that many paths. For Bernstein's mixed volume bound, this problem was solved by Huber and Sturmfels in their seminal paper \cite{HuberSturmfels95} in the form of a construction called \emph{polyhedral homotopies} (see also \cite{VerscheldeVerlindenCools94,Li99}). Since their introduction, polyhedral homotopies have become a popular default strategy in several systems such as \textsc{HomotopyContinuation.jl} \cite{HomotopyContinuation.jl}, \textsc{PHCpack} \cite{PHCpack}, and \textsc{Hom4PS} \cite{Hom4PS}. Recent work on constructing better homotopies include \cite{LeykinYu2019} (using tropical geometry), \cite{BurrSottileWalker23} (using Khovanskii bases), and \cite{DuffTelenWalkerYahl24} (using toric geometry).

In this paper, we propose a generalization of polyhedral homotopies for constructing homotopies that realize the tropical root bounds from \cite{HelminckRen22,HoltRen2023}, building on and extending ideas from Leykin and Yu \cite{LeykinYu2019}.

\subsection*{Contents of the paper and links to the existing literature}
\cref{sec:background} goes through the necessary theoretical background on parametrized polynomial systems and tropical geometry.

The main goal of \cref{sec:tropicalHomotopies} is to describe a natural generalization of polyhedral homotopies.  \cref{alg:tropicalHomotopies} describes how to construct generically optimal start systems and homotopies for solving a parametrized (Laurent) polynomial system $$\dcF=\{\dcf_1,\dots,\dcf_n\}\subseteq \CC[a_1,\dots,a_m][x_1^\pm,\dots,x_n^\pm]$$ for a generic choice $P\in(\CC^*)^m$ of parameters, using the following tropical data:
\begin{enumerate}
\item The zero-dimensional tropicalization $\Trop(\langle \dcF_Q\rangle)$, where $\dcF_Q$ is the system specialized at a perturbation $Q:=(t^{v_1}P_1,\ldots,t^{v_m}P_m)\in\CCt^m$ of the parameters, for generic exponents $v\in\QQ^m$.
\item The zeros of the initial ideals $V(\initial_w(\langle \dcF_Q\rangle))\subseteq(\CC^\ast)^n$ for all $w\in\Trop(\langle \dcF_Q\rangle)$.
\end{enumerate}
\cref{alg:tropicalHomotopies} builds on the same connection between tropical geometry and polynomial system solving that polyhedral homotopy rests on:
Consider $\dcF_Q$ as a one-parameter family of polynomial systems with parameter $t$ and variables $x$, whose specialization at $t=1$ equals a given system to be solved.  By the Newton--Puiseux theorem, solutions around $t=0$ are parametrized by Puiseux series, and consequently, their convergence or divergence at $t=0$ is governed by their coordinatewise valuations.  In \cite{HuberSturmfels95}, the coordinatewise valuations are computed using mixed cells, whereas in our work, they will be computed from tropical stable intersections.

The objective of Sections \ref{sec:verticalFamilies} and \ref{sec:horizontalFamilies} is to explain how \cref{alg:tropicalHomotopies} can be used to obtain optimal homotopies efficiently for two types of parametrized polynomial systems:

\cref{sec:verticalFamilies} considers \emph{vertically parametrized} polynomial systems and describes how to obtain their tropical data efficiently. Vertical systems arise for example from chemical reaction networks, see \cite{Dickenstein16} for a general introduction and \cite{FeliuHenrikssonPascual23} for a writeup closer to the language of this article.  Similar systems also arise from Lagrangian systems in polynomial optimization such as the maximum likelihood estimations for log-linear models, where tropical techniques have been applied in \cite{BDH24}.  We employ the idea in \cite[Section 6.1]{HelminckRen22}, explain how the tropical data above can be computed from the intersection of a tropical linear space and a tropical binomial variety, and discuss the computational challenges involved in doing so.

\cref{sec:horizontalFamilies} considers \emph{horizontally parametrized} polynomial systems as in the works of Kaveh and Khovanskii \cite{KavehKhovanskii12}.  Obtaining the required tropical data for horizontal systems is a highly non-trivial task, but we discuss two techniques that cover many systems that arise in practice:
\begin{enumerate}
\item Identifying a \emph{tropically transverse base} for the polynomial support (as in \cite[Example 6.12]{HelminckRen22} and \cite{HoltRen2023}). This is the topic of \cref{sec:horzizontalFamilyTransverse}.
\item Embedding the family into a larger family, by introducing parameters into the polynomial support. This might make the resulting root bound larger, while still being an improvement compared to the Bernstein bound. This is the topic of \cref{sec:horziontalFamilyRelaxation}.
\end{enumerate}
Both these techniques result in new, Bernstein generic systems, in such a way that an adapted version of polyhedral homotopies can be used.

The purpose of \cref{sec:caseStudies} is to demonstrate that our techniques can be applied to several examples from the existing literature:
\begin{enumerate}
\item In \cref{sec:WNT}, we examine steady state equations of the WNT pathway \cite{Gross2016wnt} by relaxing it to a vertically parametrized system.
\item In \cref{sec:Duffing}, we regard the equations for Duffing oscillators \cite{BBPMZ23} as horizontally parametrized systems with transverse base.
\item In \cref{sec:Kuramoto}, we consider the Kuramoto equations with phase delays \cite{ChenKorchevskaiaLindberg2022} as relaxed horizontally parametrized systems.
\item In \cref{sec:graphRigidity}, we study the realizations of a Laman graph \cite{cggkls}.
\end{enumerate}

In \cref{sec:conclusion}, we summarize our results and outline future research direction.

A \textsc{Julia} implementation of our algorithm based on OSCAR \cite{OSCAR} and \textsc{HomotopyContinuation.jl} \cite{HomotopyContinuation.jl}, as well as code for the examples appearing in the paper, can be found in the repository
\begin{center}
\url{https://github.com/oskarhenriksson/TropicalHomotopies.jl}\,.
\end{center}

\subsection*{Acknowledgments}
The authors would like to thank Elisenda Feliu, Máté L. Telek, and Benjamin Schr\"oter for their involvement in the early phases of this project, as well as for helpful comments and discussions. The authors also thank Paul Breiding and Sascha Timme for help with \textsc{HomotopyContinuation.jl}.

Paul Helminck was supported by the UKRI Future Leaders Fellowship ``Tropical Geometry and its applications'' (MR/S034463/2), and is supported by a JSPS Postdoctoral Fellowship (Grant No. 23769) and KAKENHI 23KF0187. Oskar Henriksson is partially supported by the Novo Nordisk project (NNF20OC0065582), as well as the European Union under the Grant Agreement~no.~101044561, POSALG. Views and opinions expressed are those of the authors only and do not necessarily reflect those of the European Union or European Research Council (ERC). Neither the European Union nor ERC can be held responsible for them. Yue Ren is supported by the UKRI Future Leaders Fellowship ``Tropical Geometry and its applications'' (MR/S034463/2 \& MR/Y003888/1).

%%% Local Variables:
%%% mode: LaTeX
%%% TeX-master: "main"
%%% ispell-local-dictionary: "en_US"
%%% End:

%% file: background.tex
\section{Background}\label{sec:background}

In this section we briefly recall some basic concepts of parametrized polynomial systems and tropical geometry that are of immediate interest to us.

\begin{convention}
  For the remainder of the article, we fix:
  \begin{enumerate}[leftmargin=*]
  \item An algebraically closed field $K$ of characteristic 0 with a possibly trivial valuation $\val\colon K^\ast\rightarrow\RR$, and residue field $\mathfrak K$. By \cite[Lemma 2.1.15]{MaclaganSturmfels15} there exists a splitting $\val(K^\ast) \rightarrow K^\ast$, which we assume to be fixed.  For example:
    \begin{enumerate}
    \item For the field of complex numbers $K=\CC$ and the trivial valuation ${\val\colon\CC^\ast\rightarrow \RR}$, a possible splitting is the map $\val(\CC^\ast)=\{0\}\rightarrow \CC^\ast$, $0\mapsto 1$,
    \item For the field of complex Puiseux series $K=\CC\{\!\{t\}\!\}$ and the usual valuation $\val\colon \CC\{\!\{t\}\!\}^\ast\rightarrow\RR$, a possible splitting is $\val(K^\ast)=\QQ \rightarrow K^\ast, \lambda\mapsto t^\lambda$.
    \end{enumerate}
    Following the notation of \cite{MaclaganSturmfels15}, we  denote the image of $\lambda \in \val(K^\ast)$ under the splitting as $t^\lambda$.
  \item An $m$-dimensional affine space $K^m$ with coordinate ring $K[a]\coloneqq K[a_1,\dots,a_m]$ and field of fractions $K(a)\coloneqq K(a_1,\dots,a_m)$. We refer to the $a$ as \emph{parameters}, $K^m$ as the \emph{parameter space} and points $P\in K^m$ as \emph{choices of parameters}.
  \item An $n$-dimensional torus $(K^\ast)^n$ with coordinate ring $K[x^\pm]\coloneqq K[x_{1}^{\pm},\dots,x_{n}^{\pm}]$.  We refer to the $x$ as \emph{variables}.
  \item We write $K^m\times (K^\ast)^n$ for the product variety with coordinate ring $K[a][x^\pm]\coloneqq K[a_1,\dots,a_m][x_1^\pm,\dots,x_n^\pm]$.
    We refer to elements $\dcf\in K[a][x^\pm]$ as \emph{parametrized (Laurent) polynomials}, ideals $\dcI\subseteq K[a][x^\pm]$ as \emph{parametrized (Laurent) polynomial ideals}, and finite sets $\{\dcf_1,\dots,\dcf_k\}\subseteq K[a][x^\pm]$ as \emph{parametrized (Laurent) polynomial systems}.
  \end{enumerate}
\end{convention}

\subsection{Parametrized polynomial systems}
In this section we recall some basic concepts of parametrized polynomial systems over algebraically closed fields.

\begin{definition}
  Let $\dcf\in K[a][x^\pm]$ be a parametrized polynomial, say $\dcf=\sum_{\alpha\in\ZZ^n}c_\alpha x^\alpha$ with $c_\alpha\in K[a]$ and $x^\alpha\coloneqq x_{1}^{\alpha_{1}}\cdots x_{n}^{\alpha_{n}}$ for $\alpha=(\alpha_{1},\dots,\alpha_{n})\in\ZZ^n$. For any choice of parameters $P\in K^m$ we define the \emph{specialization} of $\dcf$ at $P$ to be
  \begin{equation*}
    \dcf_{P}\coloneqq\sum_{\alpha\in\ZZ^n}c_\alpha(P)x^\alpha\in K[x^\pm].
  \end{equation*}
  Similarly, for a parametrized polynomial ideal $\dcI\subseteq K[a][x^\pm]$  and $P\in K^m$,  we define the \emph{specialization} of $\dcI$ at $P$ to be
  \begin{equation*}
    \dcI_{P}\coloneqq\langle \dcf_{P} : \dcf\in I \rangle\subseteq K[x^\pm].
  \end{equation*}
  Let $B_P=K[x^\pm]/\dcI_P$ be its coordinate ring.
  The \emph{root count} of $\dcI$ at $P$ is the vector space dimension $\ell_{\dcI,P}\coloneqq \dim_K(B_P)\in\ZZ_{\geq 0}\cup\{\infty\}$.
\end{definition}

\begin{remark}
  The integer $\ell_{\dcI,P}$ is the number of points in the variety $V(\dcI_{P})\subseteq(K^*)^n$ counted with a suitable multiplicity \cite[Corollary 2.5]{CoxLittleOShea05}.
  In particular, if $\dcI_{P}$ is zero-dimensional and radical, we have $\ell_{\dcI,P}=|V(\dcI_{P})|$ \cite[Corollary 2.6]{CoxLittleOShea05}.
\end{remark}

\begin{definition}
  \label{def:genericSpecialization}
  The \emph{generic specialization} of a parametrized polynomial ideal $\dcI\subseteq K[a][x^\pm]$ is the ideal $\dcI_{K(a)}\subseteq K(a)[x^\pm]$ generated by $\dcI$ under the inclusion $K[a][x^\pm]\subseteq K(a)[x^\pm]$.
  The quotient ring $B_{K(a)}\coloneqq K(a)[x^\pm]/\dcI_{K(a)}$ is a vector space over the field $K(a)$, and the \emph{generic root count} of $\dcI$ is its vector space dimension
  \[\ell_{\dcI,K(a)}:=\mathrm{dim}_{K(a)}(B_{K(a)})\in\ZZ_{\geq 0} \cup\{\infty\}. \]
  The \emph{generic dimension} of $\dcI$ is the Krull dimension of $B_{K(a)}$. If the generic dimension is zero (equivalently, if $\ell_I<\infty$), we say that $\dcI$ is \emph{generically zero-dimensional}.
  We say that $\dcI$ is \emph{generically a complete intersection} if  $B_{K(a)}\cong K(a)[z_{1},...,z_{r}]/\langle \dcf_{1},...,\dcf_{k}\rangle$ for polynomials $\dcf_1,\ldots,\dcf_k\in K(a)[z_1,\ldots,z_r]$ and $\dim(B_{K(a)})=r-k$.

  The \emph{generic root count} and \emph{generic dimension} of a parametrized polynomial system $\dcF\subseteq K[a][x^\pm]$ are the generic root count and generic dimension of the parametrized ideal it generates.  We say that $\dcF$ is \emph{generically zero-dimensional} if the parametrized ideal it generates is generically zero-dimensional.
\end{definition}

\begin{remark}\label{rem:genericProperties}\ \samepage
  \begin{enumerate}
  \item The properties in \cref{def:genericSpecialization} are generic in the sense that they reflect the behavior over a Zariski-dense open subset of $K^m$.  For example, if $\dcI$ is generically zero-dimensional, then there is a Zariski-dense open subset $U\subseteq K^m$ such that $\ell_{\dcI,P}=\ell_{\dcI,K(a)}$ for all $P\in U$ \cite[Remark 2.4]{HelminckRen22}.
  \item The generic root count $\ell_{\dcI,K(a)}$ is invariant under field extensions.  In particular, if $\dcI\subseteq \CC[a][x^\pm]$ is a parametrized polynomial ideal over the complex numbers, and $\tilde{\dcI}=\langle \dcI\rangle\subseteq \CCt[a][x^\pm]$ is the parametrized polynomial ideal over the complex Puiseux series that is generated by the elements of $\dcI$, then their generic root counts coincide, $\ell_{\dcI,\CC(a)}=\ell_{\tilde{\dcI},\CCt(a)}$.

    Consequently, from \cref{sec:tropicalHomotopies} onward, we may consider all parametrized polynomial systems over $\CC$ as parametrized polynomial systems over $\CCt$ without changing their generic root count.
  \end{enumerate}
\end{remark}

We will close this subsection with the known result that embedding para\-metrized polynomial systems can only raise the generic root count.  This is relevant for \cref{sec:horizontalFamilies}, where we embed a difficult family of polynomial system into an easier larger family of polynomial systems.  Moreover, it implies that the generic root count is an upper bound in the sense that $\ell_{\dcI,P}\leq \ell_{\dcI,K(a)}$ for any $P\in K^m$ where $\ell_{\dcI,P}$ is finite, provided $\dcI$ is generically zero-dimensional and a complete intersection.

\begin{definition}\label{def:EmbeddingSystem}
  Let $K[a][x^\pm]\coloneqq K[a_1,\dots,a_{m}][x_1^\pm,\dots,x_n^\pm]$ and $K[b][x^\pm]\coloneqq$\linebreak $K[b_1,\dots,b_{l}][x_1^\pm,\dots,x_n^\pm]$ be two parametrized polynomial rings with the same variables $x$ but different parameters $a$ and $b$. Let $\dcI_{1}\subseteq K[a][x^{\pm}]$ and $\dcI_{2}\subseteq K[b][x^{\pm}]$ be two parametrized ideals. We say $\dcI_{1}$ is \emph{embedded} in $\dcI_{2}$, if there is a ring homomorphism $K[b]\to K[a]$ such that $\dcI_1$ is the ideal generated by the image of $\dcI_2$ under the induced ring homomorphism $K[b][x^{\pm}]\to K[a][x^{\pm}]$.  Similarly, we call a system $\dcF_{1}\subseteq K[a][x^{\pm}]$ \emph{embedded} in a system $\dcF_{2}\subseteq K[b][x^{\pm}]$, if $\dcF_1$ is the image of $\dcF_2$ under the induced ring homomorphism.
\end{definition}

\begin{proposition}
  \label{pro:TheWellKnownProposition}
  Let $\dcI_{1}\subseteq K[a][x^{\pm}]$ and $\dcI_{2}\subseteq K[b][x^{\pm}]$ be two generic complete intersections.  If $\dcI_{1}$ is embedded in $\dcI_{2}$, then
  \begin{math}
    \ell_{\dcI_{1}}\leq \ell_{\dcI_{2}}.
  \end{math}
\end{proposition}
\begin{proof}
  Follows from \cite[Lemma 5.2]{HelminckRen22}.
\end{proof}

\begin{example}\label{ex:twoCircleEmbeddings}
  Consider the polynomial system $F\coloneqq \{f_1,f_2\}\subseteq \CC[x_1^\pm,x_2^\pm]$ given by
  \begin{align*}
    f_{1}=x^2+y^2+x+y+1 \quad\text{and}\quad f_{2}=3x^2+3y^2+5x+7y+11.
  \end{align*}
  The polynomials define two ellipses in the complex torus $(\CC^{*})^{2}$ intersecting in two points. \cref{fig:embeddedSystems} shows three different parametrized polynomial systems it can be embedded to. By Bernstein's Theorem, the generic root counts of $\dcF_{\mathrm{BKK}}$ is the mixed volume of the Newton polytopes, which is $4$.  One can show that the generic root counts of both $\dcF_{\mathrm{verti}}$ and $\dcF_{\mathrm{hori}}$ are $2$.  The systems $\dcF_{\mathrm{verti}}$ and $\dcF_{\mathrm{hori}}$ are examples of \emph{vertically} and \emph{horizontally parametrized system}, which are discussed in \cref{sec:verticalFamilies} and \cref{sec:horizontalFamilies}, respectively.
\end{example}

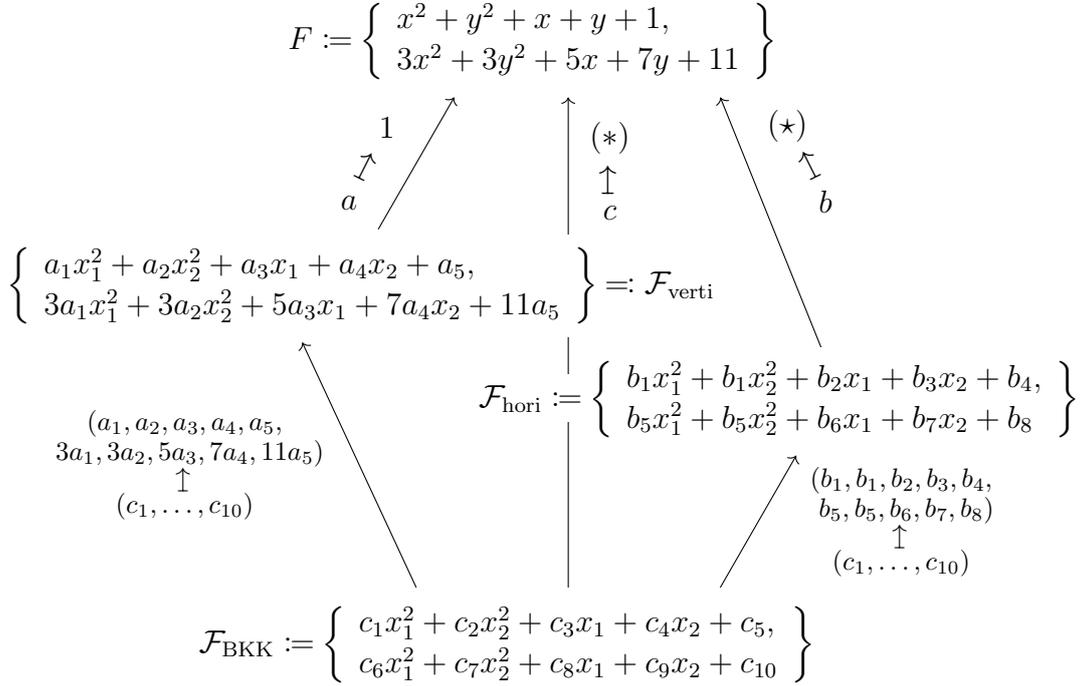
\begin{figure}[t]
  \centering
  \begin{tikzpicture}
    \draw[->] (0,-7.25) -- node[right,pos=0.85]
    {
      \begin{tikzpicture}[anchor=center]
        \node (top) at (0,0) {$(\ast)$};
        \node (bot) at (0,-1) {$c$};
        \draw[draw opacity=0] (bot) -- node[sloped]{$\mapsto$} (top);
      \end{tikzpicture}
    } (0,-0.75);
    \draw[->] (-2,-7.25) -- node[left,xshift=-2mm,font=\footnotesize]
    {
      \begin{tikzpicture}[anchor=center]
        \node at (0,0.4) {$(a_1,a_2,a_3,a_4,a_5,$};
        \node (top) at (0,0) {$\phantom{(}3a_1,3a_2,5a_3,7a_4,11a_5)$};
        \node (bot) at (0,-0.7) {$(c_1,\dots,c_{10})$};
        \draw[draw opacity=0] (bot) -- node[sloped]{$\mapsto$} (top);
      \end{tikzpicture}
    } ++(-1.5,3.25);
    \draw[->] (2,-7.25) -- node[right,xshift=4mm,font=\footnotesize]
    {
      \begin{tikzpicture}[anchor=center]
        \node at (0,0.4) {$(b_1,b_1,b_2,b_3,b_4,$};
        \node (top) at (0,0) {$\phantom{(}b_5,b_5,b_6,b_7,b_8)$};
        \node (bot) at (0,-0.7) {$(c_1,\dots,c_{10})$};
        \draw[draw opacity=0] (bot) -- node[sloped]{$\mapsto$} (top);
      \end{tikzpicture}
    } ++(1,1.75);
    \draw[->] (3.5,-4.5) -- node[right,pos=0.775]
    {
      \begin{tikzpicture}[anchor=center]
        \node (top) at (-0.5,0) {$(\star)$};
        \node (bot) at (0,-1) {$b$};
        \draw[draw opacity=0] (bot) -- node[sloped]{$\mapsfrom$} (top);
      \end{tikzpicture}
    } ++(-1.5,3.75);
    \draw[->] (-2.5,-2.5) -- node[left]
    {
      \begin{tikzpicture}[anchor=center]
        \node (top) at (0.5,0) {$1$};
        \node (bot) at (0,-1) {$a$};
        \draw[draw opacity=0] (bot) -- node[sloped]{$\mapsto$} (top);
      \end{tikzpicture}
    } ++(1,1.75);

    \node[fill=white] (polynomialSystem) at (0,0)
    {
      \begin{math}
        \left\{
          \begin{array}{l}
            x^2+y^2+x+y+1, \\
            3x^2+3y^2+5x+7y+11
          \end{array}
        \right\}
      \end{math}
    };
    \node[anchor=base east,xshift=2mm] at (polynomialSystem.base west) {$F\coloneqq$};
    \node[fill=white] (verticalSystem) at (-3.5,-3.25)
    {
      \begin{math}
        \left\{
          \begin{array}{l}
            a_1x_1^2+a_2x_2^2+a_3x_1+a_4x_2+a_5, \\
            3a_1x_1^2+3a_2x_2^2+5a_3x_1+7a_4x_2+11a_5
          \end{array}
        \right\}
      \end{math}
    };
    \node[anchor=base west,xshift=-2mm] at (verticalSystem.base east) {$\eqqcolon\dcF_{\mathrm{verti}}$};
    \node[fill=white] (horizontalSystem) at (3.5,-4.75)
    {
      \begin{math}
        \left\{
          \begin{array}{l}
            b_1x_1^2+b_1x_2^2+b_2x_1+b_3x_2+b_4, \\
            b_5x_1^2+b_5x_2^2+b_6x_1+b_7x_2+b_8
          \end{array}
        \right\}
      \end{math}
    };
    \node[anchor=base east,xshift=2mm,fill=white] at (horizontalSystem.base west) {$\dcF_{\mathrm{hori}}\coloneqq$};
    \node[fill=white] (bernsteinSystem) at (0,-8)
    {
      \begin{math}
        \left\{
          \begin{array}{l}
            c_1x_1^2+c_2x_2^2+c_3x_1+c_4x_2+c_5, \\
            c_6x_1^2+c_7x_2^2+c_8x_1+c_9x_2+c_{10}
          \end{array}
        \right\}
      \end{math}
    };
    \node[anchor=base east,xshift=2mm] at (bernsteinSystem.base west) {$\dcF_{\mathrm{BKK}}\coloneqq$};
  \end{tikzpicture}\vspace{-4mm}
  \caption{Three embeddings of the system $F$, where $(\ast)$ and $(\star)$ are suitable choices of parameters in $\CC^{10}$ and $\CC^{8}$, respectively.}\label{fig:embeddedSystems}
\end{figure}

\subsection{Tropical geometry}\label{sec:backgroundTropicalGeometry}
For tropical geometry, we will follow the notation of \cite{MaclaganSturmfels15} as closely as possible with one key difference: we tropicalize ideals instead of varieties as we are not only interested in solutions of polynomial systems, but also their multiplicity.  The resulting tropical varieties will be balanced polyhedral complexes instead of supports thereof.  Our definition of tropical varieties will therefore rely on some of the results in \cite[Sections 3.3 and 3.4]{MaclaganSturmfels15}.

\begin{definition}
  Let $\mathfrak K$ be the residue field of $K$ ($\mathfrak K=K$ if the valuation is trivial).
  The \emph{initial form} of a polynomial $f\in K[x^\pm]$, say $f=\sum_{\alpha\in S}c_\alpha  x^\alpha$ with support $S\subseteq\ZZ^n$ and coefficients $c_\alpha\in K^\ast$, with respect to a weight vector $w\in\RR^n$ is given by
  \[ \initial_w(f) \coloneqq \sum_{\substack{\alpha\in S\text{ with}\\\val(c_\alpha)+w\cdot \alpha\text{ minimal}}} \overline{t^{-\val(c_\alpha)} c_\alpha}   \, x^\alpha\in\mathfrak K[x^\pm]. \]
  The \emph{initial ideal} of an ideal $I\subseteq K[x^\pm]$ with respect to $w\in\RR^n$ is given by
  \[ \initial_w(I)\coloneqq\langle\,  \initial_w(f)\mid f\in I\, \rangle\subseteq \mathfrak K[x^\pm]. \]
\end{definition}

Defining tropical varieties as a subcomplex of the Gr\"obner complex of the homogenized ideal is quite technical.  We will omit the complete definition, refer the reader to \cite[Chapters 2 and 3]{MaclaganSturmfels15} for details, and rather explain them in some easy cases that are of importance to us.

\begin{definition}
  \label{def:tropicalVariety}
  Let $I\subseteq K[x^\pm]=K[x_1^\pm,\dots,x_n^\pm]$ be a Laurent polynomial ideal.  The \emph{tropical variety} or \emph{tropicalization} of $I$ is defined to be:
  \begin{align*}
    \Trop(I) &\coloneqq \Big\{ w\in\RR^n\mid \initial_w(I)\neq \mathfrak K[x^\pm] \Big\}.\\
    \intertext{By the Fundamental Theorem \cite[Theorem 3.2.3]{MaclaganSturmfels15}, if the valuation $\val$ is non-trivial, we also have}
    \Trop(I) &\coloneqq \mathrm{cl}\Big(\Big\{ \val(z)\in\RR^n\mid z\in V(I) \Big\}\Big)
  \end{align*}
  where $\val(\cdot)$ denotes coordinatewise valuation, and $\mathrm{cl}(\cdot)$ denotes Euclidean closure.

  By the Structure Theorem \cite[Theorem 3.3.5]{MaclaganSturmfels15}, we can use the Gr\"obner complex to make $\Trop(I)$ into a weighted polyhedral complex that, if $I$ is prime, is balanced and connected in codimension one.
\end{definition}

\begin{figure}[t]
  \centering
  \begin{tikzpicture}
    \node at (-3.5,0)
    {
      \begin{tikzpicture}[scale=1.1]
        \draw[red!70!black,very thick]
        (1,-1) -- ++(-2.5,2.5)
        (1,-1) -- node[left] {$2$} ++(0,-1.5)
        (1,-1) -- ++(2.5,2.5);
        \fill[red!70!black]
        (1,-1) circle (2.5pt);
        \node[anchor=north west,font=\scriptsize] at (1,-1) {$(-1,1)$};
        \node at (1,0.5)
        {
          \begin{tikzpicture}[scale=0.75]
            \draw (0,0) -- (2,0) -- (1,1) -- cycle;
            \fill (0,0) circle (2pt)
            (1,0) circle (2pt)
            (2,0) circle (2pt)
            (1,1) circle (2pt);
          \end{tikzpicture}
        };
        \node[red!70!black,very thick,font=\small,anchor=south west] at (-1.5,-2.5) {$\Trop(f_1)$};
        \draw[<-] (-1.25,-1) -- node[right,inner sep=2pt] {$e_2$} ++(0,0.75);
        \draw[<-] (-1,-1.25) -- node[above,inner sep=2pt] {$e_1$} ++(0.75,0);
      \end{tikzpicture}
    };
    \node at (3.5,0)
    {
      \begin{tikzpicture}[scale=1.1]
        \draw[blue!50!black,very thick]
        (-1,2) -- ++(1,1)
        (-1,2) -- ++(-1,0)
        (-1,2) -- (-1,1)
        (-1,1) -- ++(-1,0)
        (-1,1) -- (0,0)
        (0,0) -- ++(0,-1)
        (0,0) -- (1,0)
        (1,0) -- ++(0,-1)
        (1,0) -- ++(1.5,1.5);
        \fill[blue!50!black]
        (0,0) circle (2.5pt)
        (1,0) circle (2.5pt)
        (-1,1) circle (2.5pt)
        (-1,2) circle (2.5pt);
        \node[anchor=north west,font=\scriptsize,yshift=1mm] at (1,0) {$(-2,0)$};
        \node[anchor=north east,font=\scriptsize,yshift=1mm] at (0,0) {$(0,0)$};
        \node[anchor=north east,font=\scriptsize,xshift=1mm] at (-1,1) {$(2,-2)$};
        \node[anchor=south east,font=\scriptsize,xshift=1mm] at (-1,2) {$(2,-4)$};
        \node at (0.5,1.5)
        {
          \begin{tikzpicture}[scale=0.75]
            \draw (0,0) -- (2,0) -- (0,2) -- cycle
            (1,1) -- (0,1)
            (1,1) -- (0,0)
            (1,1) -- (1,0);
            \fill (0,0) circle (2pt)
            (1,0) circle (2pt)
            (2,0) circle (2pt)
            (0,1) circle (2pt)
            (1,1) circle (2pt)
            (0,2) circle (2pt);
          \end{tikzpicture}
        };
        \node[blue!50!black,very thick,font=\small,anchor=south west] at (-2,-1) {$\Trop(f_2)$};
      \end{tikzpicture}
    };
  \end{tikzpicture}\vspace{-3mm}
  \caption{Tropical hypersurfaces and Newton polytopes of the two polynomials from \cref{ex:tropicalHypersurfaces}.}
  \label{fig:tropicalHypersurfaces}
\end{figure}
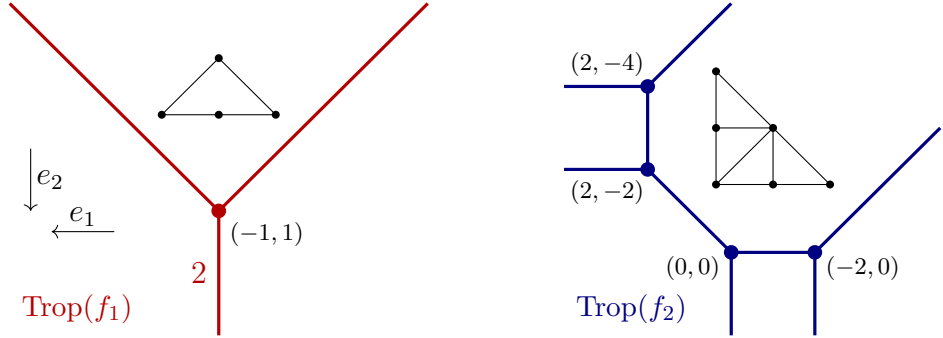

\begin{example}
  \label{ex:tropicalHypersurfaces}
  If $I=\langle f \rangle\in K[x^\pm]$ is principal, then $\Trop(f)\coloneqq\Trop(I)$ is referred to as a \emph{tropical hypersurface}, and it is dual to the regular subdivision of the Newton polytope induced by the valuation of its coefficients. The multiplicities are dual to the lattice length of the edges of the Newton polytope.

  Consider the two polynomials, whose tropical hypersurfaces and corresponding subdivisions of the Newton polytopes are illustrated in \cref{fig:tropicalHypersurfaces}:
  \begin{equation*}
    f_1=t^2x_1^2+x_1x_2+1 \;\text{and}\; f_2=t^2x_1^2+x_1x_2+t^6x_2^2+x_1+t^2x_2+1 \in \CCt[x_1^\pm,x_2^\pm].
  \end{equation*}
  The idea behind the duality is that for any $(z_1,z_2)\in(\CCt^\ast)^2$ with $f_i(z_1,z_2)=0$ the lowest $t$-degree terms in $f_i(z_1,z_2)$ must cancel.  Hence the monomials of the initial form $\initial_w(f_i)$, $w\coloneqq \val((z_1,z_2))\in\RR^2$ must form a positive-dimensional cell in the Newton subdivision.  For instance:

  $\Trop(f_1)$ has a ray $(-1,1)+\RR_{\geq 0}\cdot (0,1)$, that contains weight vectors $w\in\RR^n$ such that $\initial_w(f_1) = x_1^2+1$ and the coordinatewise valuations of solutions of the form $(t^{-1}\cdot (c+z_1'),t\cdot z_2')\in V(f_1)$, where $c\in\CC^\ast$ with $c^2=1$, and $z_i\in\CCt$ with $\val(z_i')>0$.  When substituting the solution into $f_1$, we see that the monomials of the initial form contribute to the terms of lowest $t$-degree:
  \begin{equation*}
    f_1\big(t^{-1}\cdot (c+z_1'),t\cdot z_2'\big) = \underbrace{t^2\cdot \big(t^{-1}\cdot (c+z_1')\big)^2}_{\val = 0} + \underbrace{\big(t^{-1}\cdot (c+z_1')\big)\cdot \big(t\cdot z_2'\big)}_{\val>0}+\underbrace{1\vphantom{{}_1}}_{\val=0}.
  \end{equation*}
\end{example}

\medskip

Next, we introduce stable intersections of balanced polyhedral complexes using both the definition in \cite[Definition 3.6.5]{MaclaganSturmfels15} and the equivalent formulation in \cite[Proposition 3.6.12]{MaclaganSturmfels15}.

\begin{definition}\label{def:stableIntersection}
  Let $\Sigma_1,\Sigma_2$ be two weighted balanced polyhedral complexes in $\RR^n$. Their \emph{stable intersection} is defined to be the polyhedral complex
  \[ \Sigma_1\wedge\Sigma_2\coloneqq\{\sigma_1\cap\sigma_2\mid \sigma_1\in\Sigma_1,\sigma_2\in\Sigma_2, \dim(\sigma_1+\sigma_2)=n \} \]
  with the multiplicities for the top-dimensional polyhedra given by
  \begin{equation*}
    \label{eq:stableIntersectionMultiplicities}
    \mult_{\Sigma_1\wedge\Sigma_2}(\sigma_1\cap\sigma_2) \coloneqq \sum_{\tau_1,\tau_2} \mult_{\Sigma_1}(\tau_1)\mult_{\Sigma_2}(\tau_2) [N : N_{\tau_1}+N_{\tau_2}].
  \end{equation*}
  Here, the sum is taken over all maximal $\tau_1\in\Sigma_1$ and $\tau_2\in\Sigma_2$ containing $\sigma_1\cap\sigma_2$ with $\tau_1\cap(\tau_2+\varepsilon\cdot v)\neq\emptyset$ for some generic $v\in\RR^n$ and $\varepsilon>0$ sufficiently small.  Moreover, $N$ denotes the standard lattice $\ZZ^n$, and $N_{\tau_i}$ denotes the sublattice generated by the linear span of the $\tau_i$ translated to the origin.
  Alternatively, it can be defined as:
  \[ \Sigma_1\wedge \Sigma_2 \coloneqq \lim_{\varepsilon\rightarrow 0} \Sigma_1\cap(\Sigma_2+\varepsilon\cdot v). \]
\end{definition}

\begin{example}
  \label{ex:stableIntersection}
  \cref{fig:stableIntersection} illustrates the stable intersection of $\Trop(f_1)$ and $\Trop(f_2)$ from \cref{ex:tropicalHypersurfaces}. It consists of four points, each of multiplicity $1$.
\end{example}

\begin{figure}[t]
  \centering
  \begin{tikzpicture}
    \node at (-4,0)
    {
      \begin{tikzpicture}[scale=1.1]
        \draw[blue!50!black,very thick]
        (-1,2) -- ++(1,1)
        (-1,2) -- ++(-2,0)
        (-1,2) -- (-1,1)
        (-1,1) -- ++(-2,0)
        (-1,1) -- (0,0)
        (0,0) -- ++(0,-1)
        (0,0) -- (1,0)
        (1,0) -- ++(0,-1)
        (1,0) -- ++(1.5,1.5);

        \draw[red!70!black,very thick]
        (0.85,-0.5) -- ++(-3.5,3.5)
        (0.85,-0.5) -- ++(0,-0.5)
        (0.85,-0.5) -- ++(1.65,1.65);

        \draw[fill=white]
        (-1.65,2) circle (2pt)
        (-1,1.35) circle (2pt)
        (0.35,0) circle (2pt)
        (1,-0.35) circle (2pt);

        \draw[->,densely dotted,very thick,violet]
        (0.5,-0.65) -- node[below,font=\footnotesize,xshift=-0.5mm] {$\varepsilon v$} (0.85,-0.65);
      \end{tikzpicture}
    };
    \draw[->,very thick,violet] (-0.8,-1) -- node[below] {$\varepsilon\rightarrow 0$} (0.8,-1);
    \node at (4,0)
    {
      \begin{tikzpicture}[scale=1.1]
        \draw[blue!50!black,very thick]
        (-1,2) -- ++(1,1)
        (-1,2) -- ++(-2,0)
        (-1,2) -- (-1,1)
        (-1,1) -- ++(-2,0)
        (0,0) -- ++(0,-1)
        (0,0) -- (1,0)
        (1,0) -- ++(0,-1);

        \draw[red!70!black,very thick]
        (-1,1) -- ++(-2,2)
        (0,0) -- (0.5,-0.5)
        (0.5,-0.5) -- ++(0,-0.5)
        (0.5,-0.5) -- (1,0);

        \draw[blue!50!black,very thick,dash pattern= on 6pt off 8pt,dash phase=7pt]
        (-1,1) -- (0,0)
        (1,0) -- ++(1.5,1.5);
        \draw[red!70!black,very thick,dash pattern= on 6pt off 8pt]
        (-1,1) -- (0,0)
        (1,0) -- ++(1.5,1.5);

        \draw[fill=white]
        (-2,2) circle (2pt)
        (-1,1) circle (2pt)
        (0,0) circle (2pt)
        (1,0) circle (2pt);
        \draw[decorate,decoration={brace,amplitude=5pt,mirror}] ($(-1,1)+(-0.1,-0.1)$) -- node[font=\footnotesize,anchor=north east] {intersection} ($(0,0)+(-0.1,-0.1)$);
        \node[font=\footnotesize,text width=20mm,align=center] at (0.5,1.5) {stable\\ intersection};
        \draw[<-,shorten <=5pt] (0,0) -- ++(0.25,1);
        \draw[<-,shorten <=5pt] (1,0) -- ++(-0.25,1);
        \draw[<-,shorten <=5pt] (-1,1) -- ++(0.6,0.2);
        \draw[<-,shorten <=8pt] (-2,2) to[out=337.5,in=180] ++(1.8,-0.35);
      \end{tikzpicture}
    };
  \end{tikzpicture}\vspace{-3mm}
  \caption{The stable intersection of the two tropical plane curves of \cref{ex:tropicalHypersurfaces}.}
  \label{fig:stableIntersection}
\end{figure}
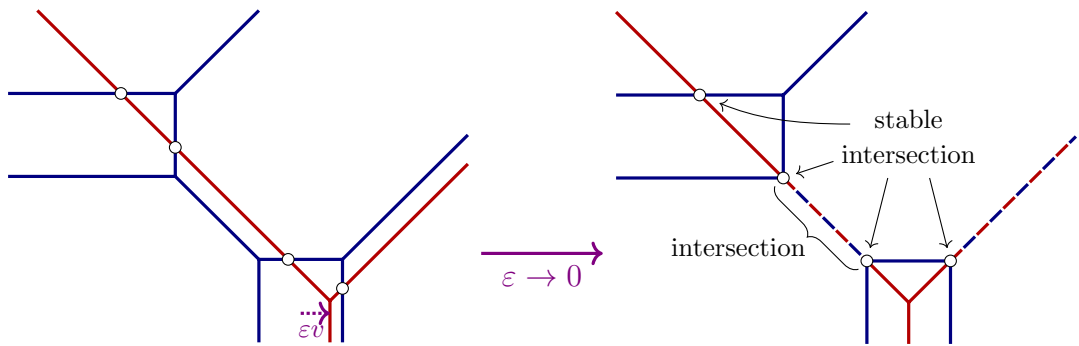

The following is a generalization of the Transverse Intersection Theorem \cite[Lemma~15]{BogartJensenSpeyerSturmfelsThomas07} from tropical varieties as \emph{supports of polyhedral complexes} to tropical varieties as \emph{balanced polyhedral complexes}.

\begin{theorem}
  \label{thm:transverseIntersection}
  Let $I,J\subseteq K[x^\pm]$ be complete intersections.  Suppose that $\Trop(I)$ and $\Trop(J)$ intersect transversally.  Then
  \[ \Trop(I + J)=\Trop(I)\wedge \Trop(J). \]
  Moreover, for all $w\in \Trop(I + J)$ we have
  \[ \initial_w(I + J) = \initial_w(I) + \initial_w(J). \]
\end{theorem}
\begin{proof}
  Set-theoretically the first statement is \cite[Theorem 3.4.12]{MaclaganSturmfels15}.  The fact that the multiplicities match follows from \cite[Corollary 5.1.3]{OssermanPayne13}, which requires $I$ and $J$ to be Cohen--Macaulay.  The latter is implied from $I$ and $J$ being complete intersections.
  The second statement on the initial ideals is proven in the proof of \cite[Theorem 3.4.12]{MaclaganSturmfels15}, see in particular \cite[Equation 3.4.3]{MaclaganSturmfels15}.
\end{proof}

We end the section with a small lemma that we need in the next section.

\begin{lemma}
  \label{lem:avoidance}
  Let $P\in(\CC^\ast)^m$, and let $X\subsetneq\CC\{\!\{t\}\!\}^m$ be proper Zariski-closed subset. Then $t^v\cdot P\notin X$ for generic $v\in\QQ^m$.
\end{lemma}
\begin{proof}
  As $X\subseteq \CCt^{m}$ is proper and Zariski closed, $X\cap(\CCt^\ast)^{m}\subseteq (\CCt^\ast)^{m}$ is also proper and Zariski closed.  Hence $\Trop(I(X))$ has positive codimension in $\RR^m$, and the com\-plement of its support is dense. This immediately gives the desired statement.
\end{proof}

%%% Local Variables:
%%% mode: LaTeX
%%% TeX-master: "main"
%%% ispell-local-dictionary: "en_US"
%%% End:

%% file: tropicalHomotopies.tex
\section{Tropical homotopies}\label{sec:tropicalHomotopies}
In this section, we explain how to construct generically optimal homotopies for a parametrized polynomial system using data about its tropicalization.  \cref{alg:tropicalHomotopies} naturally generalizes polyhedral homotopies, see \cite[Algorithm 3.1]{BBCHLS2023}, and is a variation of an idea found in the works of Leykin and Yu \cite{LeykinYu2019}.

\begin{algorithm}[Homotopies from tropical data]
  \label{alg:tropicalHomotopies}\
  \begin{algorithmic}[1]
    \REQUIRE{$(\dcF,P)$, where
      \begin{enumerate}
      \item $\dcF=\{\dcf_1,\dots,\dcf_n\}\subseteq \CC[a][x^\pm]$ is a square, parametrized polynomial system that is generically zero-dimensional and generically radical,
      \item $P\in (\CC^\ast)^{m}$.
      \end{enumerate}
      From hereon, we will regard $\dcF$ as a parametrized polynomial system over $\CC\{\!\{t\}\!\}$, and set $\dcI\coloneqq\langle\dcF\rangle\subseteq \CC\{\!\{t\}\!\}[a][x^\pm]$.
    }
    \ENSURE{A finite set $\{(H_i,V_i)\mid i=1,\dots,r\}$, where\samepage
      \begin{enumerate}
      \item $H_i\subseteq \CCt[x^\pm]$, a homotopy with $\dcI_P=\langle H_i|_{t=1}\rangle$,
      \item $V_i\subseteq V(H_i|_{t=0})\subseteq (\CC^\ast)^n$,  starting solutions,
      \end{enumerate}
      so that
      \begin{enumerate}
      \item every point in $V(\dcI_P)$ is connected to a point in $V_i$ via $H_i$ for some $i\in[r]$,
      \item $\sum_{i=1}^r |V_i| = \ell_{\dcI,\CCt(a)}$.\footnote{$|V_i|$ is counted with multiplicity as $V_i$ may not be smooth, see \cref{rem:startingSystems}.}
      \end{enumerate}
    }
    \vspace{0.2em}
    \STATE \label{algline:choiceOfValuation} Pick a generic choice of parameter valuations $v\in\QQ^{m}$ and set $Q\coloneqq t^v\cdot P\in\CC\{\!\{t\}\!\}^{m}$.\\
    \STATE \label{algline:tropicalData} Compute the tropical data required for homotopy construction:
    \begin{enumerate}
        \item $\Trop(\dcI_{Q})\subseteq\RR^n$,
        \item $V^{(w)}\subseteq(\CC^\ast)^n$, a sufficiently precise approximation of $V(\initial_w(\dcI_Q))\subseteq(\CC^\ast)^n$ for each $w\in\Trop(\dcI_Q)$.
    \end{enumerate}
    \STATE \label{algline:homotopyConstruction} Construct the homotopies
    \begin{equation*}\vspace{-1em}
      H^{(w)}{}\coloneqq \Big\{ t^{-\trop(f)(w)}\cdot f(t^{w}\cdot x)\mid f\in \dcF_Q\Big\} \subseteq\CC\{\!\{t\}\!\}[x^\pm] \text{ for all } w\in\Trop(\dcI_Q).
    \end{equation*}
    \RETURN{$\{(H^{(w)},V^{(w)}) \mid w \in\Trop(\dcI_Q)\}$}
  \end{algorithmic}
\end{algorithm}
\begin{proof}[Proof of correctness]
  Without loss of generality, we may assume that $v\in\ZZ^{m}$, so that $\dcF_Q\subseteq \CC[t^\pm][x^\pm]$.  By \cref{lem:avoidance}, $\langle\dcF_Q\rangle$ is zero-dimensional and radical.  Hence, $\dcF_Q$ may be used for homotopy continuation with target system $\dcF_Q|_{t=1}=\dcF_P$ and starting system $\dcF_Q|_{t=\varepsilon}$ for $\varepsilon>0$ sufficiently small by \cite[Theorem 7.1.1]{SommeseWampler05}.  The solutions of $\dcF_Q$ may diverge however at $t=0$.
  By the Newton--Puiseux theorem, each homotopy path is parametrized by a Puiseux series around $t=0$, and the changes of coordinates in Line~\ref{algline:homotopyConstruction} ensures that the homotopy paths whose Puiseux series have coordinatewise valuation $w$ do not diverge at $t=0$, without affecting the solutions at $t=1$.
\end{proof}

\begin{example}[Polyhedral homotopies]
  \label{ex:polyhedralHomotopies}
  Consider $F=\{f_1,f_2\}$ from \cite[Example 11]{BBCHLS2023}:
  \[ f_1\coloneqq 5-3x_1^2-3x_2^2+x_1^2x_2^2, \quad f_2\coloneqq 1+2x_1x_2-5x_1x_2^2-3x_1^2x_2\in\CC[x_1^\pm,x_2^\pm]. \]
  For the input of \cref{alg:tropicalHomotopies}, consider the parametrized system $\dcF=\{\dcf_1,\dcf_2\}$ with
  \begin{equation}
    \label{eq:bernsteinSystem}
    \begin{aligned}
      \dcf_1&\coloneqq a_{0,0}+a_{2,0}x_1^2+a_{0,2}x_2^2+a_{2,2}x_1^2x_2^2,\\
      \dcf_2&\coloneqq b_{0,0}+b_{1,1}x_1x_2+b_{1,2}x_1x_2^2+b_{2,1}x_1^2x_2\in\CC[a,b][x^\pm]
    \end{aligned}
  \end{equation}
  and parameters
  \begin{equation*}
    P\coloneqq (\underset{a_{0,0}}{\vphantom{-}5},\underset{a_{2,0}}{-3},\underset{a_{0,2}}{-3},\underset{a_{2,2}}{\vphantom{-}1},\underset{b_{0,0}}{\vphantom{-}1},\underset{b_{1,1}}{\vphantom{-}2},\underset{b_{1,2}}{-5},\underset{b_{2,1}}{-3}) \in\CC^{8},
  \end{equation*}
  so that $\dcF_P=F$.

  In Step~\ref{algline:choiceOfValuation}, consider the choice of parameter valuations
  \begin{equation*}
    v\coloneqq (\underset{a_{0,0}}{\vphantom{-}0},\underset{a_{2,0}}{\vphantom{-}0},\underset{a_{0,2}}{\vphantom{-}0},\underset{a_{2,2}}{\vphantom{-}0},\underset{b_{0,0}}{\vphantom{-}0},\underset{b_{1,1}}{\vphantom{-}2},\underset{b_{1,2}}{\vphantom{-}3},\underset{b_{2,1}}{\vphantom{-}3}) \in\QQ^{8},
  \end{equation*}
  so that $\dcf_{1,Q},\dcf_{2,Q}$ coincide with $h_1, h_2$ in \cite[Example 11]{BBCHLS2023}:
  \begin{equation*}
    \begin{aligned}
      \dcf_{1,Q} &= 5-3x_1^2-3x_2^2+x_1^2x_2^2, \\
      \dcf_{2,Q} &= 1+2t^2x_1x_2-5t^3x_1x_2^2-3t^3x_1^2x_2\in\CC\{\!\{t\}\!\}[x_1,x_2],
    \end{aligned}
  \end{equation*}

  For the tropical data in Step~\ref{algline:tropicalData}, note that by \cref{thm:transverseIntersection} and without the need for any Gr\"obner basis computations, we obtain the tropicalization $\Trop(\dcI_{Q})=\{(0,-\frac32),(-\frac32,0)\}$, see \cref{fig:polyhedralHomotopies} (a), as well as the initial ideals
  \begin{align*}
    \initial_w(\langle \dcf_{1,Q}, \dcf_{2,Q}\rangle)
    &=\initial_w(\langle \dcf_{1,Q}\rangle)+\initial_w(\langle \dcf_{2,Q}\rangle) = \langle \initial_w(\dcf_{1,Q}),\initial_w(\dcf_{2,Q})\rangle\\
    &=
      \begin{cases}
        \langle 3x_2^2+x_1^2x_2^2, 1+5x_1x_2^2 \rangle &\text{for } w=(0,-\frac32)\\
        \langle 3x_1^2+x_1^2x_2^2, 1-3x_1^2x_2\rangle &\text{for } w=(-\frac32,0)
      \end{cases}
  \end{align*}
  which yields the following $4+4$ starting solutions in $(\CC^\ast)^2$
  \begin{align*}
    V^{(0,-\frac32)} &= \Big\{ \Big(z_1, \pm \sqrt{-\tfrac{1}{5z_1}}\Big)\bigmid z_1=\pm\sqrt{-3} \Big\} \quad\text{and}\\
    V^{(-\frac32,0)} &= \Big\{ \Big(\pm \sqrt{\tfrac{1}{3z_2}},z_2\Big)\bigmid z_2=\pm\sqrt{-3} \Big\}.
  \end{align*}

  As for homotopies, for $w=(0,-\frac32)$ in Step~\ref{algline:homotopyConstruction}, we get
  \begin{equation*}
    \begin{aligned}
      \trop(\dcf_{1,Q})\Big(0,-\textstyle\frac{3}{2}\Big)&=\min\Big(0,2\cdot 0,2\cdot (-\textstyle\frac{3}{2}),2\cdot 0+2\cdot (-\textstyle\frac{3}{2})\Big)(w) = -3\\
      \trop(\dcf_{2,Q})\Big(0,-\textstyle\frac{3}{2}\Big)&=\min\Big(0,2+0+(-\textstyle\frac{3}{2}),5+0+2\cdot (-\textstyle\frac{3}{2}), 3+2\cdot 0+(-\textstyle\frac{3}{2})\Big)=0
    \end{aligned}
  \end{equation*}
  and hence
  \begin{equation*}
    \begin{aligned}
      t^{3} \dcf_{1,Q}(x_1,t^{-3/2}x_2) &= 5t^3-3t^3x_1^2-3x_2^2+x_1^2x_2^2, \\
      t^{0} \dcf_{2,Q}(x_1,t^{-3/2}x_2) &= 1+2t^{1/2}x_1x_2-5x_1x_2^2-3t^{3/2}x_1^2x_2\in\CC\{\!\{t\}\!\}[x_1,x_2].
    \end{aligned}
  \end{equation*}

  This homotopy is the same as \cite[Example 11]{BBCHLS2023} for $\alpha=(0,-3)$ after substituting $t$ by $s^2$, making all exponents integer.
  The same holds for $w=(-\frac32,0)$, which will reconstruct the homotopies of \cite[Example 11]{BBCHLS2023} for $\gamma=(-3,0)$.

  More generally, any polyhedral homotopy can be obtained from \cref{alg:tropicalHomotopies} with a parametrized input system where every coefficient has its own parameter as in \cref{eq:bernsteinSystem}.  Obtaining the $\Trop(\dcI_Q)$ for tropical data in Step~\ref{algline:tropicalData} is usually done by mixed cells, see \cref{fig:polyhedralHomotopies} (b), and obtaining the starting solutions $V^{(w)}$ is easy as the initial ideals $\initial_w(\dcI_Q)$ will be binomial.
\end{example}

\begin{figure}
  \centering
  \begin{tikzpicture}
    \node (left)
    {
      \begin{tikzpicture}
        \draw[very thick,darkblue]
        (-1,0) -- (4,0)
        (0,-1) -- (0,4);
        \fill[darkblue] (0,0) circle (2.5pt);
        \node[darkblue,anchor=south west,font=\footnotesize] at (0,0) {$(0,0)$};
        \node[below,darkblue] at (-0.65,0) {$2$};
        \node[left,darkblue] at (0,-0.65) {$2$};
        \node[below,darkblue] at (1.5,0) {$2$};
        \node[left,darkblue] at (0,1.5) {$2$};
        \node[above,darkblue] (Tropf1) at (0,4) {$\Trop(\dcf_{1,Q})$};

        \draw[thick,darkred]
        (2,2) -- ++(-3,1.5)
        (2,2) -- ++(1.5,-3)
        (2,2) -- ++(2,2);
        \fill[darkred] (2,2) circle (2.5pt);
        \node[darkred,anchor=west,yshift=-1mm,font=\footnotesize] at (2,2) {$(-1,-1)$};
        \node[darkred,above] (Tropf2) at (4,4) {$\Trop(\dcf_{2,Q})$};

        \draw[fill=white]
        (3,0) circle (3pt)
        (0,3) circle (3pt);
        \node[anchor=south west,font=\footnotesize] at (3,0) {$(-\frac{3}{2},0)$};
        \node[anchor=south west,font=\footnotesize] at (0,3) {$(0,-\frac{3}{2})$};

        \draw[<-] (1,-1.5) -- node[above]{$e_1$} ++(1,0);
        \draw[<-] (-1.5,1) -- node[right]{$e_2$} ++(0,1);
      \end{tikzpicture}
    };
    \node[anchor=north] at (left.south) {(a)};
    \node[anchor=west] (right) at (left.east)
    {
      \begin{tikzpicture}[anchor=center]
          \draw[very thick,darkblue]
          (-1,0) -- (4,0)
          (0,-1) -- (0,4);
          \fill[darkblue] (0,0) circle (3pt);

          \draw[thick,darkred]
          (2,2) -- ++(-3,1.5)
          (2,2) -- ++(1.5,-3)
          (2,2) -- ++(2,2);
          \fill[darkred] (2,2) circle (3pt);

          \draw[fill=white]
          (3,0) circle (2pt)
          (0,3) circle (2pt);

          \node at (1,1.175)
          {
            \begin{tikzpicture}[anchor=center,scale=0.325]
              % \fill[edgeViolet!20] (4,3) --  (4,1) -- (2,0) -- (2,2) -- cycle;
              % \fill[edgeViolet!20] (3,4) --  (1,4) -- (0,2) -- (2,2) -- cycle;
              \draw[thick,darkblue,fill=blue!20] (0,0) rectangle (2,2);
              \draw[thick,darkred,fill=red!20] (2,2) -- (4,3) -- (3,4) -- cycle;
              \draw[thick,darkblue]
              (4,3) -- (4,1)
              (3,4) -- (1,4);
              \draw[thick,darkred]
              (4,1) -- (2,0)
              (1,4) -- (0,2);
              \fill
              (0,0) circle (2pt)
              (2,0) circle (2pt)
              (0,2) circle (2pt)
              (2,2) circle (2pt)
              (4,1) circle (2pt)
              (1,4) circle (2pt)
              (4,3) circle (2pt)
              (3,4) circle (2pt);
              \node[font=\tiny] at (1.5,3) {4};
              \node[font=\tiny] at (3,1.5) {4};
            \end{tikzpicture}
          };
          \draw[darkblue,fill=blue!20]
          (-0.3,-0.3) rectangle (0.3,0.3);
          \node[anchor=north east,darkblue] at (-0.3,-0.3) {$\Delta(\dcf_1)$};
          \draw[darkred,fill=red!20]
          (2,2)++(-0.25,-0.25) -- ++(0.6,0.3) -- ++(-0.3,0.3) -- cycle;
          \node[anchor=south,darkred] at (2,2.4) {$\Delta(\dcf_2)$};

          \draw[fill=white]
          (3,0)++(-0.3,-0.45) -- ++(0,0.6) -- ++(0.6,0.3) -- ++(0,-0.6) -- cycle;
          \node[font=\tiny] at (3,0) {$4$};
          \draw[fill=white]
          (0,3)++(-0.45,-0.3) -- ++(0.6,0) -- ++(0.3,0.6) -- ++(-0.6,0) -- cycle;
          \node[font=\tiny] at (0,3) {$4$};
        \end{tikzpicture}
    };
    \node[anchor=north,yshift=-5.5mm] at (right.south) {(b)};
  \end{tikzpicture}\vspace{-4mm}
  \caption{The tropical intersection from \cref{ex:polyhedralHomotopies} and its mixed cells. Here, $\Delta(\cdot)$ denotes the Newton polytope. }
  \label{fig:polyhedralHomotopies}
\end{figure}
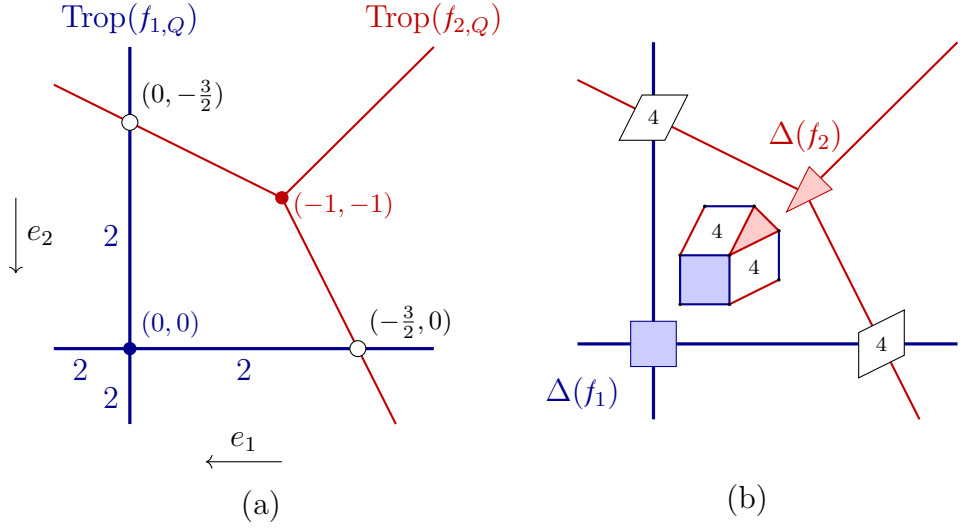

\medskip

To borrow a term from chess, we refer to the numerical challenges of initiating the tracing of the homotopies $(H^{(w)},V^{(w)})$ produced by \cref{alg:tropicalHomotopies} around $t=0$ as the \emph{early game} of path tracking.
Depending on the parametrized polynomial system $\dcF$ at hand, the following are two of the main early game issues one needs to address. (However, as we will see in later sections, none of these issues arise for the parametrized systems considered in \cref{sec:verticalFamilies,sec:horzizontalFamilyTransverse,sec:horziontalFamilyRelaxation}.)

\begin{remark}[Early game technicalities]\
  \label{rem:startingSystems}
  All potential issues with the starting system stem from the following two technicalities:
  \begin{enumerate}[label=(\roman*),leftmargin=*]
  \item \label{enumitem:startingSystemTechnicality1} the initial ideal $\initial_w(\langle \dcF_Q\rangle)\subseteq\CC[x^\pm]$ need not be radical,
  \item \label{enumitem:startingSystemTechnicality2} at $t=0$, the homotopy $H^{(w)}(0,x)\subseteq\CC[x^\pm]$ need not generate the initial ideal.
  \end{enumerate}
  Both \ref{enumitem:startingSystemTechnicality1} and \ref{enumitem:startingSystemTechnicality2} lead to the problem that, at $t=0$ and for $z\in V^{(w)}$, the Jacobian $J(H^{(w)}|_{t=0})(z)\in\CC^{n\times n}$ is not invertible, which means using the usual predictor-corrector methods described in \cite[Section 2.4]{BBCHLS2023} is not possible.

  This is a well-known issue that can be worked around by approximating the evaluation of the Puiseux series solution at small $t=\varepsilon>0$ by $z\cdot \varepsilon^w\coloneqq (z_1\cdot \varepsilon^{w_1},\dots,z_n\cdot \varepsilon^{w_n})$.  If the initial ideal $\initial_w(\langle \dcF_Q\rangle)$ is not radical, this may further require higher order terms, see \cite[Section~3.3]{sturmfels02} and \cite[Remark 6]{LeykinYu2019}.

  If $\initial_w(\langle \dcF_Q\rangle)$ is radical, we can alternatively construct the homotopies $H^{(w)}$ from the tropical Gr\"obner bases used to compute $V^{(w)}$; see \cref{ex:twoCirclesHorizontalFamily}.  If the Gr\"obner basis is not square, we can construct square homotopies by taking a generic linear combination of the Gr\"obner basis elements.
\end{remark}

The technicalities outlined in \cref{rem:startingSystems} motivate the following definition:

\begin{definition}  \label{def:genericallyBinomialInitials}
  Let $\dcI\subseteq \CC[a][x^\pm]$ be a parametrized polynomial ideal.\samepage
  \begin{itemize}
      \item  We say that $\dcI$ has \emph{binomial}, \emph{complete intersection}, or \emph{radical initials} at a choice of parameters $Q\in\CCt^{m}$, if for all $w\in\Trop(\dcI_Q)$ the initial ideal $\initial_w(\dcI_Q)$ is binomial, a complete intersection, or radical, respectively.

  \item We say $\dcI$ has one of the aforementioned properties \emph{sufficiently often}, if there is a Zariski dense set $U\subseteq \CCt^{m}$ such that the property holds for all $Q\in U$.

  \item We say $\dcI$ has one of the aforementioned properties \emph{generically}, if there is a Zariski open and dense set $U\subseteq \CCt^{m}$ such that the properties above hold for all $Q\in U$.
  \end{itemize}
\end{definition}

We conclude the section by showing why binomial initial ideals are highly desirable.  The result is used in \cref{sec:verticalFamilies}.

\begin{theorem}\label{thm:binomialInitialsAreRadical}
  Suppose that $\dcI\subseteq\CCt[a][x^\pm]$ is generically zero-dimensional.  If $\dcI$ has binomial initials at $Q\in\CCt^{m}$, then $\dcI$ has radical initials at $Q$.  In particular, if $\dcI$ has binomial initials sufficiently often, it is generically radical.
\end{theorem}
\begin{proof}
  The first part follows from \cite[Corollary 2.2]{EisenbudSturmfels1996}, which states that binomial ideals are radical over fields of characteristic $0$.

  Let $Q\in \CCt^{m}$ be a choice of parameters such that $\dcI_{Q}$ is zero-dimensional and has binomial initial ideals.  By \cite[Proposition~3.9]{HelmKatz2012}, this means that $\dcI_Q$ is sch\"on, and we can follow the argument before \cite[Proposition 3.9]{HelmKatz2012} to show that $\dcI_Q$ is radical.
  As the set of parameters $Q\in \CCt^{m}$ for which $\dcI_Q$ is radical is Zariski-open, this implies that $\dcI$ is generically radical.
\end{proof}

\begin{example}
  Many naturally occurring types of parametrized polynomial ideals $\dcI\subseteq \CC[a][x^{\pm}]$ have binomial initials sufficiently often.

  If $\mathcal{I}$ is linear in the variables $x$, then its initial ideals $\initial_w(\dcI_Q)$ are binomial provided $w\in\Trop(\dcI_Q)$ lies in the relative interior of a maximal polyhedron.  This is a consequence of \cite[Lemma 1]{BLMM17}, and the fact that all multiplicities on a tropical linear space are $1$.

  If $\dcI$ is generated by parametrized polynomials $\dcF=\{\dcf_1,\dots,\dcf_n\}$ where each coefficient is its own parameter, such as $\dcF_{\mathrm{BKK}}$ in \cref{ex:twoCircleEmbeddings}, then it is easy to find $Q\in \CCt^{m}$ such that $\dcI_Q$ has binomial initials: Simply find $Q\in \CCt^{m}$ such that
    \begin{enumerate}
    \item the valuations of the coefficients of the $\dcf_{i,Q}$ induce maximal subdivisions on the Newton polytopes,
    \item the tropicalizations $\Trop(\dcf_{i,Q})$ intersect transversally.
    \end{enumerate}
    Finding such parameters (or rather their valuations) is the first step in constructing polyhedral homotopies, see \cite[Algorithm 3.1 Line 4]{BBCHLS2023}.  The fact that the initials are binomial is a consequence of \cref{thm:transverseIntersection}.
\end{example}

%%% Local Variables:
%%% mode: LaTeX
%%% TeX-master: "main"
%%% ispell-local-dictionary: "en_US"
%%% End:

%% file: verticalModifications.tex
\section{Vertically parametrized polynomial systems}\label{sec:verticalFamilies}
In this section, we discuss vertically parametrized polynomial systems, which are inspired from the steady state equations of chemical reaction networks and certain Lagrangian systems in polynomial optimization. See \cite{FeliuHenrikssonPascual23} for a discussion on the generic geometry of vertically parametrized systems, and \cite[Section~6.1]{HelminckRen22} for a discussion on their generic root counts.
The goal of this section is to show how \cref{alg:tropicalHomotopies} can be carried out efficiently for these systems, and how they satisfy all desirable properties for it.

\begin{definition}
  \label{def:verticalFamily}
  Given a multiset $S=\{\alpha_1,\ldots,\alpha_m\}\subseteq\ZZ^n$ of exponent vectors, a \emph{vertically parametrized} system with exponents $S$ is a parametrized system
  \begin{equation*}
    \dcF\coloneqq\{\dcf_1,\dots,\dcf_n\}\subseteq \CC[a][x^\pm]\coloneqq \CC[a_1,\dots,a_m][x_1^\pm, \dots, x_n^\pm],
  \end{equation*}
  of the form
  \begin{equation*}
    \dcf_i \coloneqq \sum_{j=1}^m c_{i,j}a_j x^{\alpha_j}
  \end{equation*}
  for some coefficients $c_{i,j}\in\CC$ (note that we allow $c_{i,j}=0$).
\end{definition}

For the remainder of \Cref{sec:verticalFamilies}, we fix a vertically parametrized system $\dcF$ with some exponent vectors $S$, and let $\dcI\coloneqq\langle \dcF\rangle\subseteq\CCt[a][x^\pm]$ denote the ideal generated by $\dcF$ over $\CCt$.

\subsection{Tropical data for homotopy construction}\label{sec:verticalFamiliesTropicalData}
In this subsection, we will show that vertically parametrized systems are especially suited for \cref{alg:tropicalHomotopies}.

\begin{definition}
  \label{def:verticalModification}
  The \emph{modification} of a vertically parametrized system $\dcF$ is given by
  \begin{equation*}
    \hat \dcF\coloneqq\{\hat \dcf_i, \hat \dcg_j\mid i\in [n], j\in[m] \} \subseteq \CC[a][x^\pm,y^\pm]\coloneqq \CC[a][x_i^\pm, y_j^\pm\mid i\in[n], j\in[m]].
  \end{equation*}
  where
  \begin{equation*}
    \hat \dcf_i \coloneqq \sum_{j=1}^m c_{i,j}a_j y_j \quad\text{and}\quad \hat \dcg_j \coloneqq y_j-x^{\alpha_j}.
  \end{equation*}
\end{definition}

In what follows, we let $\hat \dcI_\lin\coloneqq\langle \hat \dcf_i\mid i\in[n]\rangle$ and $\hat \dcI_\bin\coloneqq\langle \hat \dcg_j\mid j\in[m]\rangle$ denote the ideals generated by the $\hat \dcf_i$ and $\hat \dcg_j$ respectively in $\CCt[a][x^\pm,y^\pm]$, and set $\hat \dcI\coloneqq \hat \dcI_\lin+\hat \dcI_\bin$.

This modification gives a way to compute $\Trop(\dcI_Q)$ as a stable intersection, which was also explained through the notion of toric equivariance in \cite[Section~6.1]{HelminckRen22}.

% As the intersects consist of a linear space and hypersurfaces, the intersection can be computed using tropical homotopy continuation, as discussed in \cref{rem:tropicalHomotopyContinuation}.

\begin{lemma}
  \label{lem:verticalTropicalization}
  For a generic choice of parameters $Q\in \CCt^{m}$, we have an isomorphism of (zero-dimensional) weighted polyhedral complexes
  \begin{equation}
    \label{eq:verticalModificationTropicalMap}
    \begin{array}{ccc}
      \RR^n & & \RR^n\times\RR^{m}\\
      \rotatebox[origin=c]{90}{$\subseteq$} & & \rotatebox[origin=c]{90}{$\subseteq$} \\
      \Trop(\dcI_Q)   & \overset{\cong}{\longleftrightarrow} & \Trop(\hat \dcI_{\lin,Q}) \wedge \displaystyle\bigwedge_{j=1}^m\Trop(\hat \dcg_{j,Q}) \\[7mm]
      (w_i)_{i\in [n]} & \overset{\pi}{\longmapsfrom} & \big((w_i)_{i\in[n]},(w_j)_{j\in [m]}\big)%\\[4mm]
      % (w_i)_{i\in [n]} & \longmapsto & \big((w_i)_{i\in[n]},(\sum_{i=1}^r\alpha_i\cdot w_i)_{\alpha\in S}\big).
    \end{array}
  \end{equation}
\end{lemma}
\begin{proof}
  As $\Trop(\hat \dcI_{\lin,Q})$ and the $\Trop(\hat \dcg_{j,Q})$ intersect transversally for generic $v$, we have $\Trop(\hat \dcI_{\lin,Q}) \wedge \bigwedge_{j=1}^m\Trop(\hat \dcg_{j,Q}) = \Trop(\hat \dcI_{Q})$ by \cref{thm:transverseIntersection}.  Moreover, it is straightforward to show that $\hat \dcI_Q\cap \CCt[x^\pm] = \dcI_Q$, which means that the projection in Equation~\eqref{eq:verticalModificationTropicalMap} is an isomorphism.
\end{proof}

Step \ref{algline:tropicalData} of \cref{alg:tropicalHomotopies} requires $V(\initial_{w}(\mathcal{I}_{Q}))$. For vertically parametrized systems, we can obtain generators of $\initial_{w}(\mathcal{I}_{Q})$ through a simple linear Gr\"obner basis computation.

\begin{lemma}\label{lem:verticalStartingSolutions}
  Suppose we have
  \begin{enumerate}
  \item $Q\coloneqq t^v\cdot P\in \CC\{\!\{t\}\!\}^{m}$ for some $P\in(\CC^\ast)^{m}$ and some $v\in\QQ^{m}$,
  \item \label{enumitem:verticalStartingSolutions2} $w\in \Trop(\dcI_Q)$,
  \item \label{enumitem:verticalStartingSolutions3} $\hat w\in \Trop(\hat \dcI_{\lin,Q}) \wedge \bigwedge_{j=1}^m\Trop(\hat \dcg_{j,Q})$ with $\pi(\hat w)=w$ as in \cref{lem:verticalTropicalization}, and
  \item $\hat G\subseteq\hat \dcI_{\lin,Q}$ a tropical Gr\"obner basis with respect to $\hat w$.
  \end{enumerate}
  Then $\initial_{w}(\mathcal{I}_{Q}) = \langle \initial_w(\hat g|_{y_j=x^{\alpha_j}})\mid \hat g\in \hat G\rangle$, where $(\cdot)|_{y_j=x^{\alpha_j}}$ denotes substituting all $y_j$ by $x^{\alpha_j}$.  In particular, $\{\hat g|_{y_j=x^{\alpha_j}}\}\subseteq \dcI_Q$ is a Gr\"obner basis with respect to $w$. % In particular, the generating set of $\initial_{w}(\mathcal{I}_{Q})$ is square and binomial.
\end{lemma}
\begin{proof}
  Note that
  \begin{equation*}
    \langle \initial_{\hat w}(\hat g)|_{y_j=x^{\alpha_j}}\mid \hat g\in \hat G\rangle = \langle \initial_{\hat w}(\hat g)\mid \hat g\in \hat G\rangle|_{y_j=x^{\alpha_j}} = \initial_{\hat w}(\hat \dcI_{\lin,Q})|_{y_j=x^{\alpha_j}}.
  \end{equation*}
  Hence, it suffices to show that $\initial_w(\dcI_Q)=\initial_{\hat w}(\hat \dcI_{\lin,Q})|_{y_j=x^{\alpha_j}}$.

  For the ``$\subseteq$'' inclusion, it suffices to consider elements of the form $\initial_w(g)\in\initial_w(\dcI_Q)$ for some $g=\sum_{i=1}^r q_i\dcf_{i,Q} \in \dcI_Q$ with $q_i\in\CCt[x^\pm]$.  Let $\hat g\coloneqq \sum_{i=1}^r q_i\hat \dcf_{i,Q}\in\dcI_{\lin,Q}$, so that $g=\hat g|_{y_j=x^{\alpha_j}}$.  Due to \eqref{enumitem:verticalStartingSolutions2} and \eqref{enumitem:verticalStartingSolutions3}, we then have $$\initial_w(g) = \initial_w(\hat g|_{y_j=x^{\alpha_j}})=\initial_{\hat w}(\hat g)|_{y_j=x^{\alpha_j}}\in \initial_{\hat w}(\hat \dcI_{\lin,Q})|_{y_j=x^{\alpha_j}}.$$

  For the ``$\supseteq$'' inclusion, consider $\hat g = \sum_{i=1}^r c_i \hat \dcf_{i,Q}\in \hat \dcI_{\lin,Q}$ with $c_i\in\CCt$.  As $\hat \dcI_{\lin,Q}$ is linear, it suffices to consider a linear generator $\hat g$.  Let $g\coloneqq \sum_{i=1}^r c_i  \dcf_{i,Q}$, so that $g=\hat g|_{y_j=x^{\alpha_j}}$.  Due to \eqref{enumitem:verticalStartingSolutions2} and \eqref{enumitem:verticalStartingSolutions3}, we again have \[\initial_{\hat w}(\hat g)|_{y_j=x^{\alpha_j}}=\initial_{\hat w}(\hat g|_{y_j=x^{\alpha_j}})=\initial_{w}(g)\in\initial_w(\dcI_Q).\qedhere\]
\end{proof}

\begin{example}
   Let $\dcF_{\mathrm{verti}}=\{\dcf_1,\dcf_2\}\subseteq\CC[a][x^\pm]$ be the vertically parametrized system from \cref{ex:twoCircleEmbeddings}, given by
  \begin{align*}
    \dcf_1 &= \phantom{3}a_1 x_1^2 + \phantom{3}a_2 x_2^2 + \phantom{5}a_3 x_1 + \phantom{7}a_4 x_2 + \phantom{11}a_5 \quad \text{and}\\
    \dcf_2 &= 3a_1 x_1^2 + 3a_2 x_2^2 + 5a_3 x_1 + 7a_4 x_2 + 11a_5,
  \end{align*}
  and consider its modification
  \begin{align*}
    \hat \dcf_1 &= \phantom{3}a_1 y_1 + \phantom{3}a_2 y_2 + \phantom{5}a_3 y_3 + \phantom{7}a_4 y_4 + \phantom{11}a_5 y_5, & \hat \dcg_1 &= y_1-x_1^2, & \hat \dcg_3 &= y_3-x_1. \\
    \hat \dcf_2 &= 3a_1 y_1 + 3a_2 y_2 + 5a_3 y_3 + 7a_4 y_4 + 11a_5y_5 & \hat \dcg_2 &= y_2-x_2^2,  & \hat \dcg_4 &= y_4-x_2, \\
    & & & & \hat \dcg_5 &= y_5-1.
  \end{align*}
  (Note that the introduction of $y_3, y_4, y_5$ is not strictly necessary, as their monomials are linear or constant. They can be omitted as a computational optimization.)

  For $Q\coloneqq(t, 1, 1, t, 1)\in\CCt^{m}$, we obtain $\Trop(\hat\dcI_{\lin,Q})$ and $\Trop(\hat\dcI_{\bin,Q})$ that intersect transversally in a single point $\hat w=0\in\RR^{7}$ of multiplicity $2$.  By \cref{lem:verticalTropicalization}, this means $\Trop(I_Q)=\{(0,0)\}$, and for $w=(0,0)$, \cref{alg:tropicalHomotopies} Line \ref{algline:homotopyConstruction} gives the homotopy
  \begin{align*}
    H^{(w)} = \Big( &t^0\cdot \big(\phantom{3}t^{1-0} x_1^2 + \phantom{3}t^{0-0} x_2^2 + \phantom{5}t^{0-0} x_1 + \phantom{7}t^{1-0} x_2 + \phantom{11}t^{0}\big), \\
                    &t^0\cdot\big( 3t^{1-0} x_1^2 + 3t^{0-0} x_2^2 + 5t^{0-0}x_1 + 7t^{1-0}x_2 + 11t^{0}\big)\Big)\\
    =\Big( &tx_1^2+x_2^2+x_1+tx_2+1,\;\; 3tx_1^2 + 3x_2^2 + 5x_1 + 7tx_2 + 11\Big).
  \end{align*}
  Note that $H^{(w)}=(\dcf_{1,Q},\dcf_{2,Q})$ due to $w=(0,0)$.  At $t=0$, we have the equations $x_2^2+x_1+1=0=3x_2^2+5x_1+11$, which are not binomial.  In order to obtain binomial equations, we can compute a (tropical) Gr\"obner basis of $\dcI_Q$ with respect to $w$ using \cref{lem:verticalStartingSolutions}, which is:
  \begin{equation*}
    g_1 \coloneqq x_1+tx_2+4 \quad\text{and}\quad g_2\coloneqq 4tx_1^2 + 4x_2^2 + 3x_1 + 2tx_2
  \end{equation*}
  Using these instead of the $\dcf_{i,Q}$ in \cref{alg:tropicalHomotopies} Line \ref{algline:homotopyConstruction} gives us a homotopy $H^{(w)}=(g_1,g_2)$ that is binomial at $t=0$.
\end{example}

We conclude this subsection, by noting that vertically parametrized polynomial systems exhibit all desirable properties for \cref{alg:tropicalHomotopies} that are discussed at the end of \cref{sec:tropicalHomotopies}.

\begin{theorem}
  \label{thm:verticalSystems}
  Let $\dcF$ be a generically zero-dimensional vertically parametrized system. Then $\dcF$ generically has binomial, complete intersection, and radical initials.
\end{theorem}
\begin{proof}
  Consider the generating sets $\{\initial_{\hat w}(\hat g|_{y_j=x^{\alpha_j}})\mid \hat g\in \hat G\}\subseteq\initial_w(\dcI_Q)$ of \cref{lem:verticalStartingSolutions}.  The complete intersection property follows from the fact that the cardinality of a (minimal) Gr\"obner basis $\hat G$ of a linear ideal $\hat \dcI_{\lin,Q}$ always equals their codimension, which is $n$.  To show binomiality, observe that for $Q=t^v\cdot P$ with $v\in\RR^n$ generic we may assume that $\hat w$ lies in the relative interior of a maximal polyhedron of $\Trop(\hat \dcI_{\lin,Q})$.  Hence the $\initial_{\hat w}(\hat \dcI_{\lin,Q})$ is binomial by \cite[Lemma 1]{BLMM17}, and therefore its Gr\"obner basis $\hat G$ can be chosen to be binomial.  Radicality then follows from \cref{thm:binomialInitialsAreRadical}.
\end{proof}

See  \cite{FeliuHenrikssonPascual23} for an alternative proof of generic radicality, as well as explicit conditions for when a vertically parametrized system is generically zero-dimensional.

\begin{remark}
  \label{rem:tropicalCompleteIntersection}
  In some cases, the $n$-codimensional tropical linear space $\Trop(\hat \dcI_{\lin,Q})$ is a \emph{tropical complete intersection}, namely
  \begin{equation}
    \label{eq:tropicalCompleteIntersection}
    \Trop(\hat \dcI_{\lin,Q})=\Trop(\hat h_1)\wedge\dots\wedge\Trop(\hat h_n) \text{ for linear } \hat h_1,\dots,\hat h_n\in\hat\dcI_{\lin,Q}.
  \end{equation}
  By \cite[Theorem 3.6.1]{MaclaganSturmfels15}, Condition~\eqref{eq:tropicalCompleteIntersection} holds if and only if the column matroid of the coefficient matrix $(c_{i,j}Q_j)_{i\in [n], j\in [m]}\in\CCt^{n\times m}$ is \emph{transversal}; see \cite[Section 2.2]{Bonin2010} for a definition. The latter can be tested in \textsc{polymake} \cite{polymake}, and, if the matroid is transversal, the transversal presentation that is computed during the test can be used to construct the $\hat h_i$.

  If Condition~\eqref{eq:tropicalCompleteIntersection} holds, then one can show that the generic root count of the original system $\dcF$ is the mixed volume of the Newton polytopes of the $\hat h_i|_{y_j=x^{\alpha_j}}\in\CCt[x^\pm]$:
  \begin{align*}
    & \ell_{\dcI,\CC(a)} \overset{\text{\cref{lem:verticalTropicalization}}}{=} \ell_{\hat \dcI,\CC(a)} \overset{\text{\cite[Theorem 4.6.8]{MaclaganSturmfels15}}}{=} \MV(\Delta(\hat h_i), \Delta(\hat \dcg_j)\mid i\in [n], j\in [m]) \\
    & \hspace{5mm}= \MV(\Delta(\hat h_i|_{y_j=x^{\alpha_j}}) \mid i\in [n]),
  \end{align*}
  where $\Delta({\cdot})$ denotes the Newton polytope, $\MV(\cdot)$ denotes the mixed volume, and the final equality follows from the fact that the mixed volume is invariant under toric reembeddings.
  Instead of using \cref{alg:tropicalHomotopies}, one can now simply construct polyhedral homotopies for the $\hat h_i|_{y_j=x^{\alpha_j}}\in\CCt[x^\pm]$. Note that the mixed volume of the $\hat h_i|_{y_j=x^{\alpha_j}}$ need not be the mixed volume of $\dcF_Q$.
\end{remark}

\subsection{Remarks on computational ingredients}
We close this section with a few remarks on some computations that are required for obtaining the necessary tropical data for \cref{alg:tropicalHomotopies} via \cref{lem:verticalTropicalization} and \cref{lem:verticalStartingSolutions}.

\begin{remark}[Tropical intersections]
  \label{rem:tropicalHomotopyContinuation}
  \cref{lem:verticalTropicalization} requires the computation of
  \begin{equation}
    \label{eq:tropicalIntersection}
    \Trop(\hat \dcI_{\lin,Q}) \wedge \displaystyle\bigwedge_{j=1}^m\Trop(\hat \dcg_{j,Q}).
  \end{equation}
  Constructing $\Trop(\hat \dcI_{\lin,Q})$ and intersecting it with $\bigwedge_{j=1}^m\Trop(\hat \dcg_{j,Q})$ are both known to be difficult tasks individually.  However, computing the intersection in Equation~\eqref{eq:tropicalIntersection} can be done faster than the sum of its constituencies.
  In \cite{Jensen16}, Jensen develops a tropical homotopy continuation approach for computing transverse intersection of $n$ hypersurfaces in $\RR^n$.  The resulting number is the mixed volume of the polynomials of the hypersurfaces \cite[Theorem 4.6.8]{MaclaganSturmfels15}.  Both \textsc{Gfan} \cite{gfan} and \textsc{HomotopyContinuation.jl} \cite{HomotopyContinuation.jl}, and by proxy also \textsc{Singular} \cite{Singular} and \textsc{Macaulay2} \cite{M2}, rely on it for mixed volume computations.

  Similar to homotopy continuation for polynomial systems, the basic idea of tropical homotopy continuation is to deform an easy starting intersection to a desired target intersection.  What makes this process possible in tropical geometry is the duality between tropical hypersurfaces and regular subdivisions of Newton polytopes.  As tropical linear spaces are dual to matroid subdivisions of matroid polytopes, this approach can in principle be generalized to intersections of hypersurfaces and tropical linear spaces as required in \cref{lem:verticalTropicalization}.  The first steps were done in \cite{DaiseyRen2024}.
\end{remark}

\begin{remark}[Tropical Gr\"obner bases]
  \label{rem:tropicalGroebnerBases}
  \cref{lem:verticalStartingSolutions} requires the computation of a set $G\subseteq I$ such that $\initial_w(I)=\langle \initial_w(g)\mid g\in G\rangle$ for an ideal $I\subseteq\CCt[x^\pm]$ and $w\in\Trop(I)$.  Such $G$ are also known as tropical Gr\"obner bases \cite[Section 2.4]{MaclaganSturmfels15}.
  Similar to classical Gr\"obner bases, they can be computed using Buchberger-like algorithms \cite{ChanMaclagan18,MarkwigRen20} or using \texttt{F4-/F5-}like algorithms \cite{Vaccon18}.

  Most importantly for our purposes, if $I$ is generated by linear polynomials, we can compute $G$ using a single Gaussian elimination on the Macaulay matrix of its generators \cite[Algorithm 3.2.2]{Vaccon18}.  Hence, we will generally consider tropical Gr\"obner bases of linear ideals a computational non-issue.
\end{remark}

\begin{remark}[Binomial systems]
  \label{rem:binomialSystemSolving}
  Homotopy continuation generally requires starting solutions for the path tracking.  Naturally, these starting solutions should come from systems that are easy to solve.  The starting systems produced by \cref{lem:verticalStartingSolutions} are binomial, similar to polyhedral homotopies.  Binomial systems are easily solvable as, modulo a change of coordinates that can be computed using a Smith normal form on the exponent matrix (see, e.g., \cite{ChenLi14}), any binomial system is of the form
  \[ x_1^{d_1}-c_1 = 0,\qquad \dots,\qquad x_n^{d_n}-c_n=0. \]
\end{remark}

%%% Local Variables:
%%% mode: LaTeX
%%% TeX-master: "main"
%%% ispell-local-dictionary: "en_US"
%%% End:

%% file: horizontalModifications.tex
\section{Horizontally parametrized polynomial systems}\label{sec:horizontalFamilies}
In this section, we discuss horizontally parametrized polynomial systems, which were prominently studied by Kaveh and Khovanskii \cite{KavehKhovanskii12} and many others using the theory of Newton--Okounkov bodies.  We show that the ideas from \cref{sec:verticalFamilies} are insufficient for addressing general systems of such type, and focus on two related types of parametrized polynomial systems instead:  One is a particular class of horizontally parametrized systems, the other is a relaxation of horizontally parametrized systems, i.e., a larger parametrized family of polynomial systems as in \cref{pro:TheWellKnownProposition}.

\begin{definition}
  \label{def:horizontalFamily}
   A \emph{horizontally parametrized} system with polynomial support $R=\{q_1,\ldots,q_m\}\subseteq\CC[x^\pm]$ is a parametrized system
  \begin{equation*}
    \dcF\coloneqq\{\dcf_1,\dots,\dcf_n\}\subseteq \CC[a][x^\pm]\coloneqq \CC[a_{i,j}\mid i\in[n], j\in [m]][x_1^\pm, \dots, x_n^\pm],
  \end{equation*}
  of the form
  \begin{equation}
    \label{eq:horizontalSystem}
    \dcf_i \coloneqq \sum_{j=1}^m c_{i,j}a_{i,j}q_j
  \end{equation}
  for some coefficients $c_{i,j}\in\CC$ (note that we allow $c_{i,j}=0$).
\end{definition}

For the remainder of~\cref{sec:horizontalFamilies}, we fix a horizontally parametrized system $\dcF\subseteq\CC[a][x^\pm]$ with polynomial support $R = \{q_1,\dots,q_m\} \subseteq \CC[x^\pm]$. We furthermore let $\dcI\coloneqq\langle \dcF\rangle \subseteq\CCt[a][x^\pm]$ be the ideal generated by $\dcF$.

It is commonplace to only consider solutions in a Zariski open space that depends on the polynomial support $R$ \cite[Definition 4.5]{KavehKhovanskii12}.   In \cref{lem:horizontalTropicalization} below, this is done by saturation.  Note that such systems indeed satisfy the requirements of \cref{alg:tropicalHomotopies}, see \cite[Theorem 4.9]{KavehKhovanskii12} or \cite[Proposition 6.9]{HelminckRen22}.

\subsection{The challenge of horizontally parametrized systems}\label{sec:horizontalFamiliesTropicalData}
In \cite{LeykinYu2019}, Leykin and Yu suggest considering the following modification.

\begin{definition}
  \label{def:horizontalModification}
  The \emph{modification} of a horizontally parametrized system $\dcF$ is
  \begin{equation*}
    \hat \dcF\coloneqq\{\hat \dcf_i, \hat \dcg_j \mid i\in [n], j\in[m] \} \subseteq \CC[a][x^\pm,y^\pm]\coloneqq \CC[a][x_i^\pm, y_j^\pm\mid i\in[n], j\in [m]]
  \end{equation*}
  where
  \begin{equation*}
    \hat f_i \coloneqq \sum_{j=1}^m c_{i,j}a_{i,j}y_j \quad\text{and}\quad \hat g_j \coloneqq y_j-q_j.
  \end{equation*}
\end{definition}

We will use the notation $\hat \dcI_\lin\coloneqq\langle \hat \dcf_i\mid i\in[n]\rangle$ and $\hat \dcI_\nlin\coloneqq\langle \hat \dcg_j\mid j\in[m]\rangle$ for the ideals generated by the $\hat \dcf_i$ and $\hat \dcg_j$ respectively in $\CCt[a][x^\pm,y^\pm]$, and set $\hat \dcI\coloneqq \hat \dcI_\lin+\hat \dcI_\nlin$.

As in \cref{lem:verticalTropicalization}, the tropicalization of the horizontal modification also decomposes and relates to the tropicalization of the original ideal.  The following lemma is an extension of \cite[Lemma 4]{LeykinYu2019}.

\begin{lemma}
  \label{lem:horizontalTropicalization}
  For a generic choice of parameters $Q\in \CCt^{m}$, we have an isomorphism of (zero-dimensional) weighted polyhedral complexes
  \begin{equation}
    \label{eq:horizontalModificationTropicalMap}
    \begin{array}{ccc}
      \RR^n & & \RR^n\times\RR^{R}\\
      \rotatebox[origin=c]{90}{$\subseteq$} & & \rotatebox[origin=c]{90}{$\subseteq$} \\
      \Trop(\dcI_Q:(\prod_{j=1}^m q_{j})^\infty)   & \overset{\cong}{\longleftrightarrow} & \Big(\displaystyle\bigwedge_{i\in [n]} \Trop(\hat \dcf_{i,Q})\Big) \wedge \Trop(\hat \dcI_{\nlin,Q}) \\[7mm]
      (w_i)_{i\in [n]} & \longmapsfrom & \big((w_i)_{i\in[n]},(w_q)_{q\in R}\big).
    \end{array}
  \end{equation}
\end{lemma}
\begin{proof}
  As the $\Trop(\hat \dcf_{i,Q})$ and $\Trop(\hat \dcI_{\nlin,Q})$ intersect transversally for generic $v$, we have $(\bigwedge_{i\in [n]} \Trop(\hat \dcf_{i,Q})) \wedge \Trop(\hat \dcI_{\nlin,Q}) = \Trop(\hat \dcI_{Q})$ by \cref{thm:transverseIntersection}.  To show that the projection in Equation~\eqref{eq:horizontalModificationTropicalMap} is an isomorphism, we need to show that there is an isomorphism
  \[ \bigslant{\CCt[x^\pm]}{\dcI_Q:(\prod_{j=1}^m q_{j})^\infty} \overset{\sim}{\longrightarrow} \bigslant{\CCt[x^\pm,y^\pm]}{\hat \dcI_Q},\quad \overline x\longmapsto \overline x.  \]

  To show that the map is well-defined, let $h\in \CCt[x^\pm]$ with $h\cdot q^\beta\in\dcI_Q$ for some $\beta\in\ZZ_{\geq 0}^m$.  As $\dcI_Q\subseteq\hat\dcI_Q$ as sets, we have $h\cdot q^\beta\in\hat\dcI_Q$, which implies $h\cdot y^\beta\in\hat\dcI_Q$ and consequently $h\in\hat\dcI_Q$.

  The map is clearly surjective, as $\overline x$ maps to $\overline x$, and $\overline q_j$ maps to $\overline q_j = \overline y$.

  To show that the map is injective, let $h\in \hat \dcI_Q\cap\CCt[x^\pm]$, say
  \begin{align*}
    h &= \Big(\sum_{i=1}^n h_{1,i}\cdot \hat \dcf_{i,Q}\Big) + \Big(\sum_{j=1}^m h_{2,j}\cdot \hat \dcg_{j,Q}\Big)\hspace{-8mm}&&\text{for some }h_{1,i}, h_{2,j}\in\CCt[x^\pm,y^\pm]\\
    \intertext{Substituting $y_j$ by $q_j$ on both sides (leaving the left side unchanged) then yields}
    h &= \Big(\sum_{i=1}^n h_{1,i}'\cdot \dcf_{i,Q}\Big)\hspace{-8mm} &&\text{for some }h_{1,i}'\in\CCt[x^\pm]
  \end{align*}
  showing that $h\in\dcI_Q$.
\end{proof}

\begin{example}
  \label{ex:twoCirclesHorizontalFamily}
  Consider $\dcF_{\mathrm{hori}}=\{\dcf_1,\dcf_2\}\subseteq\CC[b][x^\pm]$ from \cref{ex:twoCircleEmbeddings} given by
  \begin{equation*}
    \dcf_1\coloneqq b_1x_1^2+b_1x_2^2+b_2x_1+b_3x_2+b_4 \quad\text{and}\quad \dcf_2\coloneqq b_5x_1^2+b_5x_2^2+b_6x_1+b_7x_2+b_8,
  \end{equation*}
  which was to be solved for $P=(1,1,1,1,3,5,7,11)$, and its modification
  \begin{align*}
    \hat \dcf_1&\coloneqq b_1y_1+b_2y_2+b_3y_3+b_4y_4, & \hat \dcg_1&\coloneqq y_1-(x_1^2+x_2^2), & \hat\dcg_3&\coloneqq y_3-x_2,\\
    \hat \dcf_2&\coloneqq b_5y_1+b_6y_2+b_7y_3+b_8y_4, & \hat \dcg_2&\coloneqq y_2-x_1,& \hat\dcg_4&\coloneqq y_4-1.
  \end{align*}
  Note that all $\Trop(\hat\dcg_{j,Q})$ are intersecting transversally, which means $$\Trop(\hat\dcI_{\nlin,Q}) = \bigwedge_{j=1}^4\Trop(\hat\dcg_{j,Q}),$$ making it easy to compute.  Moreover, the introduction of $y_2, y_3, y_4$ is not strictly necessary as $q_2, q_3, q_4$ are variables or constants.  They can be omitted for the sake of optimization.

  For $Q\coloneqq (t^3, 1, t^2, t^3, 3t^2, 5t^2, 7t^3, 11t^2)$, we obtain $\Trop(\hat\dcI_{\lin,Q})$ and $\Trop(\hat\dcI_{\nlin,Q})$ that intersect transversally in a single point $\hat w=(2,0,0,2,0,0)$ of multiplicity $2$, as illustrated in \cref{fig:horizontalModificationIntersection}.
  By \cref{lem:horizontalTropicalization}, this means $\Trop(\dcI_Q)=\{(2,0)\}$, and for $w=(2,0)$ \cref{alg:tropicalHomotopies} Line~\ref{algline:homotopyConstruction} gives the homotopy
  \begin{align*}
    H^{(w)}=\Big( &t^{-2}\cdot \big(\phantom{3}t^{3+4}x_1^2 + \phantom{3}t^{3+0}x_2^2 + \phantom{5}t^{0+2}x_1 + \phantom{7}t^{2+0}x_2 + \phantom{11}t^3\big), \\
                  &t^{-2}\cdot\big( 3t^{2+4}x_1^2 + 3t^{2+0}x_2^2 + 5t^{2+2}x_1 + 7t^{3+0}x_2 + 11t^2\big)\Big)\\
           =\Big( &t^5x_1^2 + tx_2^2 + x_1 + x_2 + t,\;\; 3t^4x_1^2 + 3x_2^2 + 5t^2x_1 + 7tx_2 + 11\Big).
  \end{align*}
  At $t=0$, we have the binomial equations $x_1+x_2=0=3x_2^2+11$, which has two solutions, and at $t=1$ we obtain the target system $\dcF_{\mathrm{hori}}|_{b=1}$.

  In contrast, for $Q\coloneqq (1,t,t^2,1,1,t,t^2,1)$, we obtain $\Trop(\hat\dcI_{\lin,Q})$ and $\Trop(\hat\dcI_{\nlin,Q})$ that intersect transversally in a single point $w=-(1,1,0,1,1,0)$ of multiplicity $2$.
  By \cref{lem:horizontalTropicalization}, this means $\Trop(\dcI_Q)=\{-(1,1)\}$.  And for $w=-(1,1)$, Line~\ref{algline:homotopyConstruction} in \cref{alg:tropicalHomotopies} gives the homotopy
  \begin{align*}
    H^{(w)}=\Big( &t^2\cdot \big(\phantom{3}t^{0-2}x_1^2 + \phantom{3}t^{0-2}x_2^2 + \phantom{5}t^{1-1}x_1 + \phantom{7}t^{2-1}x_2 + 1\big), \\
                  &t^2\cdot\big( 3t^{0-2}x_1^2 + 3t^{0-2}x_2^2 + 5t^{1-1}x_1 + 7t^{2-1}x_2 + 11\big)\Big)\\
           =\Big( &x_1^2 + x_2^2 + t^2x_1 + t^3x_2 + t^2,\;\; 3x_1^2 + 3x_2^2 + 5t^2x_1 + 7t^3x_2 + 11t^2\Big).
  \end{align*}
  At $t=0$, we have the binomial equations $x_1^2 + x_2^2=0=3x_1^2 + 3x_2^2$, which is problematic as they cut out a one-dimensional solution set.  However the system has only two solutions for $t>0$ sufficiently small, and one can show that those solutions converge to the two solutions of $\initial_w(\dcI_Q)=\langle -2x_1-8, 3x_1^2 + 3x_2^2\rangle$ as $t$ goes to $0$.  See \cref{rem:startingSystems} for more details on homotopy continuation under such circumstances.

  Alternatively, consider the following two polynomials that form a (tropical) Gr\"obner basis of $\dcI_Q$ with respect to $w$:
  \begin{equation*}
    g_1 \coloneqq -2tx_1 - 4t^2x_2 - 8 \quad\text{and}\quad  g_2\coloneqq 3x_1^2 + 5tx_1 + 3x_2^2 + 7t^2x_2 + 11.
  \end{equation*}
  Using them instead of the $\dcf_{i,Q}$ in \cref{alg:tropicalHomotopies} Line~\ref{algline:homotopyConstruction} gives us
  \begin{align*}
    H^{(w)}=\Big( &t^0\cdot \big(-2t^{1-1}x_1 -4t^{2-1}x_2 - 8 \big), \\
                  &t^2\cdot\big( 3t^{0-2}x_1^2 + 3t^{0-2}x_2^2 + 5t^{1-1}x_1 + 7t^{2-1}x_2 + 11\big)\Big)\\
    =\Big( &-2x_1 -4tx_2 - 8,\;\; 3x_1^2 + 3x_2^2 + 5t^2x_1 + 7t^3x_2 + 11t^2\Big).
  \end{align*}
  At $t=0$, we obtain a binomial generating set of $\initial_w(\dcI_Q)$ used above.
\end{example}

\begin{figure}[t]
  \centering
  \begin{tikzpicture}
    \node (centerHypersurface) at (0,0)
    {
      \begin{tikzpicture}[x={(240:7mm)},y={(180:7mm)},z={(90:7mm)},every node/.style={font=\small}]
        % hypersurface
        % x = x1
        % y = x2
        % z = y
        \coordinate (w1) at (1.25,1.25,2.5); % two ends of the line
        \coordinate (w2) at (-1.25,-1.25,-2.5);
        % \fill[blue,opacity=0.2] (w1) -- ($(w1)+(2,0,0)$) -- ($(w2)+(2,0,0)$) -- (w2) -- cycle;
        % \fill[blue,opacity=0.2] (w1) -- ($(w1)+(0,2,0)$) -- ($(w2)+(0,2,0)$) -- (w2) -- cycle;
        % \fill[blue,opacity=0.2] (w1) -- ($(w1)+(0,0,2)$) -- ($(w2)+(0,0,2)$) -- (w2) -- cycle;
        \fill[blue,opacity=0.2] (w1) -- ($(w1)+(-0.333,1.667,-0.667)$) -- ($(w2)+(-0.333,1.667,-0.667)$) -- (w2) -- cycle;
        \draw (w1) -- ($(w1)+(-0.333,1.667,-0.667)$) -- ($(w2)+(-0.333,1.667,-0.667)$) -- (w2) -- cycle;
        \fill[blue!20] (w1) -- ($(w1)+(1.667,-0.333,-0.667)$) -- ($(w2)+(1.667,-0.333,-0.667)$) -- (w2) -- cycle;
        \draw (w1) -- ($(w1)+(1.667,-0.333,-0.667)$) -- ($(w2)+(1.667,-0.333,-0.667)$) -- (w2) -- cycle;
        \fill[blue,opacity=0.2] (w1) -- ($(w1)+(-1,-1,1)$) -- ($(w2)+(-1,-1,1)$) -- (w2) -- cycle;
        \draw (w1) -- ($(w1)+(-1,-1,1)$) -- ($(w2)+(-1,-1,1)$) -- (w2) -- cycle;
        \draw[very thick,darkblue]
        (-1.5,-1.5,-3) -- (1.5,1.5,3) node[anchor=south,xshift=-4mm] {$\RR\cdot (1,1,2)$};
        \draw[darkblue,->] (0,0,0) -- (3,0,0) node[anchor=north] (ex1) {$e_1$};
        \draw[darkblue,loosely dotted] (0,0,0) -- (0,2.25,0);
        \draw[darkblue,->] (0,2.25,0) -- (0,3,0) node[anchor=east] (ex2) {$e_2$};
        \draw[darkblue,->] (0,0,0) -- (0,0,3.5) node[anchor=south] (ey1) {$e_3$};
        % \draw[darkblue,dashed] (0,0,0) -- (2.5,-0.5,-1);
        % \draw[darkblue,dashed] (0,0,0) -- (-0.5,2.5,-1);
        % \draw[darkblue,dashed] (0,0,0) -- (-1.5,-1.5,1.5);
        \draw[fill=white]
        plot[mark=*,mark size=2.5pt] (1.5,0,0);
        \draw[fill=white]
        plot[mark=square*,mark size=2.5pt] (-1,-1,0);
      \end{tikzpicture}
    };
    \node[anchor=north,yshift=2mm,font=\small] (centerHypersurfaceText) at (centerHypersurface.south) {$\Trop(\dcI_{\nlin,Q})$};

    \node (leftCurve) at (-5,0)
    {
      \begin{tikzpicture}[x={(240:4mm)},y={(180:4mm)},z={(90:4mm)},every node/.style={font=\small}]
        % linear space
        % x = x1
        % y = x2
        % z = y
        \coordinate (v1) at (1,-1,0);
        \coordinate (v2) at (3,1,0);
        \draw[darkred,->]
        (v1) -- ++(-1.5,-1.5,-1.5) node[anchor=west] {$e_0$};
        \draw[darkred,->]
        (v1) -- ++(0,0,2) node[anchor=south] {$e_3$};
        \draw[darkred]
        (v1) -- (v2);
        \draw[darkred,->]
        (v2) -- ++(2,0,0) node[anchor=north] {$e_1$};
        \draw[darkred,->]
        (v2) -- ++(0,2,0) node[anchor=east] {$e_2$};
        \fill[darkred]
        (v1) circle (2.5pt)
        (v2) circle (2.5pt);
        \node[anchor=south east,darkred] at (v1) {$(1,-1,0)$};
        \node[anchor=north west,darkred] at (v2) {$(3,1,0)$};
        \draw[fill=white]
         plot[mark=*,mark size=2.5pt] (2,0,0);
      \end{tikzpicture}
    };
    % \node[anchor=north,text width=55mm,align=center,font=\small] (leftCurveText) at (leftCurve.south)
    % {$\Trop(\dcI_{\lin,Q})$\\ $Q=(t^3, 1, t^2, t^3, t^2, t^2, t^3, t^2)$};
    \node[xshift=-50mm,font=\small] (leftCurveText) at (centerHypersurfaceText) {$\Trop(\dcI_{\lin,Q})$};
    \node[anchor=north,yshift=-0.5mm,font=\tiny] (leftCurveParam) at (leftCurveText.336) {$(t^3, 1, t^2, t^3, 3t^2, 5t^2, 7t^3, 11t^2)$};
    \node at (leftCurveParam.north) {\rotatebox{90}{$=$}};

    \node (rightCurve) at (5,0)
    {
      \begin{tikzpicture}[x={(240:4mm)},y={(180:4mm)},z={(90:4mm)},every node/.style={font=\small}]
        % linear space
        % x = x1
        % y = x2
        % z = y
        \coordinate (v) at (-1,-2,0);
        \draw[darkred,->]
        (v) -- ++(-1.5,-1.5,-1.5) node[anchor=west] {$e_0$};
        \draw[darkred,->]
        (v) -- ++(0,0,2) node[anchor=south] {$e_3$};
        \draw[darkred,->]
        (v) -- ++(2,0,0) node[anchor=north] {$e_1$};
        \draw[darkred,->]
        (v) -- ++(0,2,0) node[anchor=east] {$e_2$};
        \fill[darkred]
        (v) circle (2.5pt);
        \node[anchor=south west,darkred] at (v) {$(-1,-2,0)$};
        \draw[fill=white] plot[mark=square*,mark size=2.5pt] (-1,-1,0);
      \end{tikzpicture}
    };
    \node[xshift=50mm,font=\small] (rightCurveText) at (centerHypersurfaceText) {$\Trop(\dcI_{\lin,Q})$};
    \node[anchor=north,yshift=-0.5mm,font=\tiny] (rightCurveParam) at (rightCurveText.336) {$(1,t,t^2,1,3,5t,7t^2,11)$};
    \node at (rightCurveParam.north) {\rotatebox{90}{$=$}};
  \end{tikzpicture}
  \caption{The transverse intersections of \cref{ex:twoCirclesHorizontalFamily} illustrated in $\RR^3\cong\{e_{y_2}=e_{x_1},e_{y_3}=e_{x_2},e_{y_4}=0\}\subseteq \RR^{6}$ with $e_1\coloneqq e_{x_1}$, $e_2\coloneqq e_{x_2}$, $e_3\coloneqq e_{y_1}$, and $e_0\coloneqq -e_1-e_2-e_3$.}
  \label{fig:horizontalModificationIntersection}
\end{figure}

Unlike \cref{lem:verticalTropicalization} however, no fast approach is known for computing the intersection in \cref{lem:horizontalTropicalization}.  Unlike the binomial ideal $\hat \dcI_{\bin,Q}$ of \cref{lem:verticalTropicalization}, the non-linear ideal $\hat \dcI_{\nlin,Q}$ in \cref{lem:horizontalTropicalization} has no easily exploitable structure in general.  Computing $\Trop(\hat \dcI_{\nlin,Q})$ using current algorithms would require several Gr\"obner basis computations \cite{BJSST2007,MarkwigRen20}.

\begin{example}\label{ex:harderHorizontalSystem}
  Consider the following horizontally parametrized system that has generic root count $3$, polynomial support $R=\{(1 + x_1 + x_2)^3,(1 + x_1 + x_2)^2,x_1,1\}$, and suppose we want to solve it for $P=(1,1,1,1,2,3,5,7)\in \CC^{8}$:
  \begin{equation}
    \label{eq:horizontalSystemHard}
    \begin{aligned}
      \dcf_1 &= a_1 (1 + x_1 + x_2)^3 + a_2 (1 + x_1 + x_2)^2 + a_3x_1 + a_4 \\
      \dcf_2 &= a_5 (1 + x_1 + x_2)^3 + a_6 (1 + x_1 + x_2)^2 + a_7x_1 + a_8.
    \end{aligned}
  \end{equation}
  Its modification is given by
  \begin{align*}
    \hat\dcf_1 &= a_1y_1 + a_2y_2 + a_3y_3 + a_4y_4, & \hat\dcg_1 &= y_1-(1+x_1+x_2)^3,
    & \hat\dcg_3&=y_3-x_1,\\
    \hat\dcf_2 &= a_5y_1 + a_6y_2 + a_7y_3 + a_8y_4, & \hat\dcg_2 &= y_2-(1+x_1+x_2)^2,  & \hat\dcg_4&=y_4-1.
  \end{align*}
  We will continue this example in the upcoming sections.
\end{example}

\subsection{Supports with tropically transverse base}\label{sec:horzizontalFamilyTransverse}
One way to deal with the difficulty of horizontally parametrized systems is to impose extra conditions on the polynomial support. In \cite{HoltRen2023}, the authors focus on the following special case of horizontally parametrized systems.

\begin{definition}
[{\cite[Definition 9]{HoltRen2023}}]
  \label{def:transverseHorizontalFamily}
  We say the polynomial support $R=\{q_1,\dots,q_m\}\subseteq\CC[x^\pm]$ has a \emph{tropically transverse base}, if there are sets $S=\{b_1,\dots,b_l\}\subseteq\CC[x^\pm]$ and $B=\{\beta_1,\dots,\beta_l\}\subseteq\ZZ^m$ such that
  \begin{enumerate}
  \item $q_j=b^{\beta_j}=\prod_{k=1}^lb_k^{\beta_{j,k}}$ for $j\in [m]$,
  \item $\Trop(b_1), \dots, \Trop(b_l)$ intersect transversally.
  \end{enumerate}
\end{definition}

\begin{definition}[{\cite[Definition 10]{HoltRen2023}}]
  \label{def:transverseHorizontalModification}
  Let $\dcF$ have a polynomial support with tropically transverse base $S$, as in \cref{def:transverseHorizontalFamily}.  The \emph{(two-stage) modification} of $\dcF$ is defined to be
  \begin{align*}
    \hat\dcF &\coloneqq\{\hat\dcf_i,\hat\dcg_j,\hat\dch_k\mid i\in[n], j\in[m], k\in[l]\}\\
    &\hspace{15mm}\subseteq\CC[a][x^\pm,y^\pm,z^\pm]\coloneqq\CC[a][x_i^\pm,y_j^\pm,z_k^\pm\mid i\in[n], j\in[m], k\in[l]]
  \end{align*}
  where
  \begin{equation}
    \label{eq:transverseHorizontalModification}
    \hat \dcf_i \coloneqq \sum_{j=1}^m c_{i,j}a_j z_j,\quad \hat \dcg_j \coloneqq z_j-\prod_{k=1}^ly_k^{\beta_{j,k}},\quad \text{and}\quad \hat \dch_k \coloneqq y_k-b_k.
  \end{equation}
\end{definition}

  For the remainder of this subsection, we let $\hat \dcI\coloneqq\langle \hat\dcF \rangle\subseteq\CCt[a][x^\pm,y^\pm,z^\pm]$ be the parametrized ideal generated by $\hat\dcF$.

\begin{theorem}[{\cite[Theorem 2]{HoltRen2023}}]
  \label{thm:transverseHorizontalModification}
  With the notation above,
  \[ \ell_{\hat \dcI,\CC(a)} = \MV\Big(\Delta(\dcf_i),\Delta(\dcg_j),\Delta(\dch_k)\mid i\in[n], j\in[m], k\in[l]\Big). \]
\end{theorem}

As a corollary of \cref{thm:transverseHorizontalModification}, polyhedral homotopies are optimal for System~\eqref{eq:transverseHorizontalModification}.  Furthermore, the polyhedral homotopies for the modified System~\eqref{eq:transverseHorizontalModification} can be reformulated to homotopies for the original System~\eqref{eq:horizontalSystem} by back-substituting the modification variables, as illustrated by the next example.

\begin{example}
  \label{ex:horizontalSystemHardTwoStageModification}
  The polynomial support of System~\eqref{eq:horizontalSystemHard} in \cref{ex:harderHorizontalSystem} has the tropically transverse base
  \[S=\{1+x_1+x_2,x_1\},\quad B=\{(3,0),(2,0),(0,1),(0,0)\}.\]
  Its two-stage modification is thus given by
  \begin{equation}
    \label{eq:horizontalSystemHardTwoStageModification}
    \begin{aligned}
      \hat\dcf_1 &= a_1z_1 + a_2z_2 + a_3z_3 + a_4z_4, & \hat\dcg_1 &= z_1-y_1^3,
      & \hat \dch_1 &= y_1-(1+x_1+x_2),\\
      \hat\dcf_2 &= a_5z_1 + a_6z_2 + a_7z_3 + a_8z_4, & \hat\dcg_2 &= z_2-y_1^2, & \hat \dch_2 &= y_2-x_1.\\
                 & & \hat\dcg_3&=z_3-y_2,\\
                 & & \hat\dcg_4&=z_4-1.
    \end{aligned}
  \end{equation}
  By \cref{thm:transverseHorizontalModification}, the generic root count of the modified system is its mixed volume, and polyhedral homotopies are optimal for solving it.

  Recall from \cref{ex:polyhedralHomotopies} that polyhedral homotopies require choosing valuations for all coefficients. For System~\eqref{eq:horizontalSystemHardTwoStageModification}, it actually suffices to simply choose $Q=(t^4, t^2, 1, 1, 2t^{11}, 3t^7, 5, 7t)$, which yields six tropical hypersurfaces that intersect transversally in two points:
  \begin{align*}
    &\Big(\bigwedge_{i=1}^2\Trop(\hat\dcf_{i,Q})\Big)\wedge \Big(\bigwedge_{j=1}^4 \Trop(\hat\dcg_{j,Q})\Big)\wedge \Big(\bigwedge_{k=1}^2 \Trop(\hat\dch_{k,Q})\Big)\\
    &\qquad = \Big\{ (1,-1,-1,1,-3,-2,1,0), (1,-2,-2,1,-6,-4,1,0) \Big\}.
  \end{align*}
  For $w=(1,-1,-1,1,-3,-2,1,0)$, the resulting homotopies from \cref{alg:tropicalHomotopies} Line~\ref{algline:homotopyConstruction} are:
  \begin{equation*}
    \begin{aligned}
      \hat h_1^{(w)} &= tz_1 + z_2 + tz_3 + z_4, & \hat h_3^{(w)} &= z_1-y_1^3,
      & \hat h_7^{(w)} &= y_1-(t+t^2x_1+x_2),\\
      \hat h_2^{(w)} &= 2t^7z_1 + 3t^4z_2 + 5z_3 + 7z_4, & \hat h_4^{(w)} &= z_2-y_1^2, & \hat h_8^{(w)} &= y_2-x_1.\\
                 & & \hat h_5^{(w)}&=z_3-y_2,\\
                 & & \hat h_6^{(w)}&=z_4-1.
    \end{aligned}
  \end{equation*}
  Note that we can undo the modification by substituting the $y$ and the $z$ to obtain a homotopy involving only the $x$:
  \begin{align*}
    h_1^{(w)} &= t(t+t^2x_1+x_2)^3 + (t+t^2x_1+x_2)^2 + tx_1 + 1,\\
    h_2^{(w)} &= 2t^7(t+t^2x_1+x_2)^3 + 3t^4(t+t^2x_1+x_2)^2 + 5x_1 + 7,
  \end{align*}
  which for $t=0$ yields a binomial system with $2$ solutions in $(\CC^\ast)^2$ (the latter may be seen easier by specializing the $\hat h_i^{(w)}$ at $t=0$).

  Similarly, for $w=(1,-2,-2,1,-6,-4,1,0)$, we obtain a homotopy that for $t=0$ yields a binomial system with $1$ solution in $(\CC^\ast)^2$.  The total number of starting solutions therefore equals the generic root count of System~\eqref{eq:horizontalSystemHard}, which is $3$.
\end{example}

\subsection{Support relaxation}\label{sec:horziontalFamilyRelaxation}
Another way to deal with the difficulty of horizontally parametrized systems, besides imposing extra conditions as in \cref{sec:horzizontalFamilyTransverse}, is by embedding them into a larger and easier parametrized family in the sense of \cref{pro:TheWellKnownProposition}.  For instance, the following relaxation ensures that the modification is Bernstein generic.

\begin{definition}
  \label{def:horizontalSystemRelaxed}
  Suppose that the polynomial support $R=\{q_1,\dots,q_m\}$ has the form
  \begin{math}
    q_j=\sum_{k=1}^l q_{j,k}x^{\alpha_k}\text{ with } q_{j,k}\in\CC
  \end{math}
  for some finite set of exponent vectors $S\coloneqq \{\alpha_1,\dots,\alpha_l\}\subseteq\ZZ^n$ and some coefficients $q_{j,k}\in\CC$ (note that we may have $q_{j,k}=0$).
  The \emph{relaxation} of $\dcF$ is the parametrized system
  \begin{equation*}
    \dcF^\sharp\coloneqq\{\dcf_1^\sharp,\dots,\dcf_n^\sharp\}\subseteq \CC[a,b][x^\pm]\coloneqq \CC[a_{i,j},b_{j,k}\mid i\in[n], j\in [m], k\in [l]][x_1^\pm, \dots, x_n^\pm],
  \end{equation*}
  where
  \begin{equation*}
    \dcf_i^\sharp \coloneqq \sum_{j=1}^m c_{i,j}a_{i,j} \underbrace{\sum_{k=1}^l q_{j,k}b_{j,k} x^{\alpha_k}}_{\eqqcolon \dcq_j^\sharp}.
  \end{equation*}
\end{definition}

We use $\dcI^\sharp\subseteq\CCt[a,b][x^\pm]$ to denote the ideal generated by $\dcF^\sharp$.

\begin{definition}
  \label{def:horizontalSystemRelaxedModification}
  The \emph{modification} of $\dcF^\sharp$ is the system
  \begin{equation*}
    \hat \dcF^\sharp\coloneqq\{\hat \dcf_i^\sharp, \hat \dcg_j^\sharp\mid i\in [n], j\in [m]\} \subseteq \CC[a,b][x^\pm,y^\pm]\coloneqq \CC[a,b][x_i^\pm, y_j^\pm\mid i\in[n], j\in [m]].
  \end{equation*}
  where
  \begin{equation*}
      \hat \dcf_i^\sharp \coloneqq \sum_{j=1}^m c_{i,j}a_{i,j}y_j \quad\text{and}\quad \hat \dcg_j^\sharp \coloneqq y_j-\dcq_j^\sharp.
  \end{equation*}
\end{definition}

We use $\hat \dcI_\lin^\sharp$ and $\hat \dcI_\nlin^\sharp$ to denote the ideals in $\CCt[a,b][x^\pm,y^\pm]$ generated by the $\hat \dcf_i^\sharp$ and $\hat \dcg_j^\sharp$, respectively, and set $\hat \dcI^\sharp = \hat \dcI_\lin^\sharp + \hat \dcI_\nlin^\sharp$.

\begin{corollary}
  \label{cor:horizontalSystemRelaxedModification}
  We have $\ell_{\hat \dcI^\sharp,\CC(a)}=\MV(\hat \dcf_i^\sharp,\hat \dcg_j^\sharp \mid i\in[n], j\in[m] )$.
\end{corollary}
\begin{proof}
  This follows directly from \cref{thm:transverseIntersection}.
\end{proof}

\begin{example}
  \label{ex:horizontalSystemHardRelaxed}
  The relaxed horizontal modification for the horizontal system form \Cref{ex:harderHorizontalSystem} is given by
  \begin{equation}
    \label{eq:horizontalSystemHardRelaxed}
    \begin{aligned}
      \hat\dcf_1 &= a_1y_1 + a_2y_2 + a_3y_3 + a_4y_4,\\
      \hat\dcf_2 &= a_5y_1 + a_6y_2 + a_7y_3 + a_8y_4,\\
      \hat\dcg_1 &= y_1-(b_{1,1}x_1^3 + 3b_{1,2}x_2x_1^2 + 3b_{1,3}x_1^2 + 3b_{1,4}x_2^2x_1 + 6b_{1,5}x_2x_1 + 3b_{1,6}x_1\\
                 &\hspace{17mm} + b_{1,7}x_2^3 + 3b_{1,8}x_2^2 + 3b_{1,9}x_2 + b_{1,10}),\\
      \hat\dcg_2 &= y_2-(b_{2,1}x_1^2 + 2b_{2,2}x_2x_1 + 2b_{2,4}x_1 + b_{2,5}x_2^2 + 2b_{2,6}x_2 + b_{2,7}),\\
      \hat\dcg_3&=y_3-b_{3,1}x_1,\\
      \hat\dcg_4&=y_4-b_{4,1}.
    \end{aligned}
  \end{equation}
  By \cref{cor:horizontalSystemRelaxedModification}, the generic root count of the relaxed modification is equal to its mixed volume, which is $6$, exceeding the generic root count of the original System~\eqref{eq:horizontalSystemHard} but staying below the original mixed volume of $9$.  Similar as in \cref{ex:horizontalSystemHardTwoStageModification}, we can construct polyhedral homotopies for System~\eqref{eq:horizontalSystemHardRelaxed} and undo the modification by substituting the $y$ to obtain a homotopy purely in the $x$, tracing only $6$ paths instead of the $9$ paths of polyhedral homotopies.
\end{example}

We conclude this section with a short proof that the generic root count of relaxed system never exceeds the mixed volume of the original system, i.e., that the homotopies that one obtains from the polyhedral homotopies of the modified system will never involve more paths than the polyhedral homotopies of the original system.

\begin{proposition}
  Let $\dcI$ be the horizontally parametrized ideal from \cref{def:horizontalFamily}, let $\hat\dcI$ be its modification from \cref{def:horizontalModification}, let $\dcI^\sharp$ be its relaxation from \cref{def:horizontalSystemRelaxed}, and let $\hat \dcI^\sharp$ be its relaxed modification from \cref{def:horizontalSystemRelaxedModification}. Then the generic root counts satisfy the  following chain of inequalities:
  \[\ell_{\dcI:(\prod_{i=1}^m q_i)^\infty,\CC(a)}=\ell_{\hat\dcI,\CC(a)}\leq\ell_{\hat\dcI^\sharp,\CC(a,b)}\leq\ell_{\dcI^\sharp,\CC(a,b)}\leq\MV(\dcF). \]
\end{proposition}
\begin{proof}
  The first equation follows from \cref{lem:horizontalTropicalization}, whereas the first and third inequality follow from \cref{pro:TheWellKnownProposition}.
  For the second inequality, notice that for all $Q\in(\CC^{\ast})^{m}$ there is an isomorphism
  \begin{equation*}
    \bigslant{K[x^\pm]}{\dcI^\sharp_Q\colon \big(\prod_{j=1}^m\dcq_{j,Q}\big)^\infty} \longrightarrow K[x^\pm,y^\pm] / \hat \dcI_Q^\sharp, \qquad \overline x \longmapsto \overline x,
  \end{equation*}
  which can be proven with the same arguments as in the proof of \cref{lem:horizontalTropicalization}, and that by definition root counts can only decrease with saturation.  In other words, for all $Q\in(\CC^{\ast})^{m}$ the root count of $\hat \dcI_Q^\sharp$ is smaller or equal to that of $\dcI_Q^\sharp$.
\end{proof}

\begin{remark}
  Using \cite[Proposition 5.16]{HelminckRen22}, one can also show that if $\dcI$ is generically zero-dimensional, then $\hat \dcI^\sharp$ has the same generic root count as $\dcI^\sharp$.
\end{remark}

%%% Local Variables:
%%% mode: LaTeX
%%% TeX-master: "main"
%%% ispell-local-dictionary: "en_US"
%%% End:

%% file: caseStudies.tex
\section{Case Studies}\label{sec:caseStudies}
In this section, we revisit several examples of parametrized systems in existing literature, and show how our techniques can be used in their context.  Notebooks with computations for the case studies can be found in the repository
\begin{center}
    \url{https://github.com/oskarhenriksson/TropicalHomotopies.jl}\,.
\end{center}

\subsection{Steady states of chemical reaction networks}\label{sec:WNT}
Consider the chemical reaction network of the WNT pathway as studied in \cite{Gross2016wnt}:
\begin{equation*}
  \begin{array}{l@{\hspace{2cm}}l}
    \ce{X_{1} <=>[$k_1$][$k_2$] X_{2}} & \ce{X_{3} + X_{6} <=>[$k_{14}$][$k_{15}$] X_{15} ->[$k_{16}$] X_{3} + X_{7}} \\
    \ce{X_{2} + X_{4} <=>[$k_3$][$k_4$] X_{14} ->[$k_5$] X_{2} + X_{5}} & \ce{X_{7} + X_{9} <=>[$k_{17}$][$k_{18}$] X_{17} ->[$k_{19}$] X_{6} + X_{9}} \\
    \ce{X_{5} + X_{8} <=>[$k_6$][$k_7$] X_{16} ->[$k_8$] X_{4} + X_{8}} & \ce{X_{6} + X_{11} <=>[$k_{20}$][$k_{21}$] X_{19} ->[$k_{22}$] X_{6} + \emptyset} \\
    \ce{X_{4} + X_{10} <=>[$k_9$][$k_{10}$] X_{18} ->[$k_{11}$] X_{4} + \emptyset} & \ce{X_{11} ->[$k_{23}$] \emptyset} \\
    \ce{{\emptyset} ->[$k_{12}$] X_{10}} & \ce{X_{11} + X_{12} <=>[$k_{24}$][$k_{25}$] X_{13}} \\
    \ce{X_{10} ->[$k_{13}$] \emptyset} & \ce{X_{2} <=>[$k_{26}$][$k_{27}$] X_{3}} \\
    \ce{X_{5} <=>[$k_{28}$][$k_{29}$] X_{7}} & \ce{X_{10} <=>[$k_{30}$][$k_{31}$] X_{11}}.
  \end{array}
\end{equation*}
The steady states of chemical reaction networks under mass action kinetics can be described by systems that are close to being vertically parametrized (in \cite{FeliuHenrikssonPascual23} these systems are referred to as \emph{augmented} vertically parametrized systems, see also \cite{RoseTelek24} for a discussion from the point of view of positive tropicalizations).

For the WNT pathway, we obtain a square parametrized polynomial system with $19$ variables $x$ and $36$ parameters $k,c$:
\allowdisplaybreaks
\begin{align}
  \nonumber \dcf_1 &= k_{1} x_{1} - (k_{2}+ k_{26}) x_{2} + k_{27} x_{3} - k_{3} x_{2} x_{4} +  (k_{4}+k_{5}) x_{14} \\
  \nonumber \dcf_2 &= k_{26} x_{2} - k_{27} x_{3} - k_{14} x_{3} x_{6} + (k_{15}+ k_{16}) x_{15} \\
  \nonumber \dcf_3 &= -k_{28} x_{5} + k_{29} x_{7} - k_{6} x_{5} x_{8} + k_{5} x_{14} + k_{7} x_{16} \\
  \nonumber \dcf_4 &= -k_{14} x_{3} x_{6} - k_{20} x_{6} x_{11} + k_{15} x_{15} +  k_{19} x_{17} + (k_{21} + k_{22}) x_{19} \\
  \nonumber \dcf_5 &= k_{28} x_{5} - k_{29} x_{7} - k_{17} x_{7} x_{9} + k_{16} x_{15} +  k_{18} x_{17} \\
  \nonumber \dcf_6 &= k_{12} - (k_{13}+k_{30}) x_{10} - k_{9} x_{4} x_{10} + k_{31} x_{11} +  k_{10} x_{18} \\
  \nonumber \dcf_7 &= -k_{23}x_{11} + k_{30} x_{10} - k_{31} x_{11} - k_{20} x_{6} x_{11} - k_{24} x_{11} x_{12} +   k_{25} x_{13} + k_{21} x_{19} \\
  \nonumber \dcf_8 &= k_{24} x_{11} x_{12} - k_{25} x_{13} \\
  \nonumber \dcf_9 &= k_{3} x_{2} x_{4} - (k_{4} + k_{5}) x_{14} \\
  \label{eq:wnt} \dcf_{10} &= k_{14} x_{3} x_{6} - (k_{15} + k_{16}) x_{15} \\
  \nonumber \dcf_{11} &= -k_{6} x_{5} x_{8} + (k_{7} + k_{8}) x_{16} \\
  \nonumber \dcf_{12} &= k_{17} x_{7} x_{9} - (k_{18} - k_{19}) x_{17} \\
  \nonumber \dcf_{13} &= k_{9} x_{4} x_{10} - (k_{10} + k_{11}) x_{18} \\
  \nonumber \dcf_{14} &= k_{20} x_{6} x_{11} - (k_{21} + k_{22}) x_{19} \\
  \nonumber \dcf_{15} &= x_1+x_2+x_3+x_{14}+x_{15} - c_1 \\
  \nonumber \dcf_{16} &= x_4+x_5+x_6+x_7+x_{14}+x_{15}+x_{16}+x_{17}+x_{18}+x_{19} - c_2 \\
  \nonumber \dcf_{17} &= x_8+x_{16} - c_3 \\
  \nonumber \dcf_{18} &= x_9+x_{17} - c_4 \\
  \nonumber \dcf_{19} &= x_{12}+x_{13} - c_5.
\end{align}
The mixed volume of System~\eqref{eq:wnt} is 56, but the generic root count is only $9$ \cite[Theorem~1.1]{Gross2016wnt}.  We can embed System~\eqref{eq:wnt} into a vertically parametrized family by introducing new parameters for all monomials of the so-called \emph{conservation laws} $\dcf_{15},\dots,\dcf_{19}$ that lack parametric coefficients, and \cref{alg:tropicalHomotopies} will construct homotopies with exactly $9$ paths (see the GitHub repository for examples of such homotopies).
The general fact that this type of embedding does not increase the generic root count for any steady state system is the subject of an upcoming preprint.

The main bottleneck is the computation of the tropical intersection data for \Cref{alg:tropicalHomotopies} via \Cref{lem:verticalTropicalization}.  It consists of:
\begin{enumerate}
\item \label{enumitem:bottleneck1} the computation of $\Trop(\hat\dcI_{Q,\lin})$, which consists of 78\,983 maximal polyhedra,
\item \label{enumitem:bottleneck2} the intersection $\Trop(\hat\dcI_{Q,\lin})\wedge \Trop(\hat\dcI_{Q,\lin})$, which consists of $9$ points counted with multiplicity.
\end{enumerate}
The bottlenecks~\eqref{enumitem:bottleneck1} and~\eqref{enumitem:bottleneck2} required $43$ seconds and $21$ seconds, respectively, on a MacBook~Air with an Apple M2 chip and 16 GB of RAM.
We note that the matroid associated with $\Trop(\hat\dcI_{Q,\lin})$ is non-transversal, so \cref{rem:tropicalCompleteIntersection} does not apply.

\subsection{Periodic solutions to Duffing oscillators}\label{sec:Duffing}
In this section, we consider equations arising from coupled driven non-linear oscillators as discussed by Borovik, Breiding, del Pino, Micha{\l}ek, and Zilberberg in their work on semimixed systems of polynomial equations \cite{BBPMZ23}.  To be precise, we consider the following parametrized system in \cite[Section 6.2]{BBPMZ23} with variables $u_i,v_i$ and parameters $\omega_i$, $\alpha$, $\gamma$, $\lambda$:
\begin{equation*}
  \label{eq:oscillators}
  \begin{aligned}
    \left(\frac{\omega_{0}^{2}-\omega_1^{2}}{2}-\frac{\lambda\omega_{0}^{2}}{4}\right)u_1 + \frac{\gamma \omega_1}{2} v_1 + \frac{3}{8} \alpha u_1(u_1^2 + v_1^2) + \frac{3}{4} \alpha u_1(u_2^2+v_2^2)& = 0, \\
    \left(\frac{\omega_{0}^{2}-\omega_1^{2}}{2}+\frac{\lambda\omega_{0}^{2}}{4}\right)v_1 - \frac{\gamma \omega_1}{2} u_1 +\frac{3}{8} \alpha v_1(u_1^2 + v_1^2) +\frac{3}{4}\alpha v_1(u_2^2+v_2^2)& = 0,\\
    \frac{\left(\omega_{0}^{2}-\omega_2^{2}\right)}{2}u_2 + \frac{\gamma \omega_2}{2}v_2 +\frac{3}{8} \alpha u_2(u_2^2 + v_2^2) + \frac{3}{4} \alpha u_2(u_1^2+v_1^2)& = 0,\\
    \frac{\left(\omega_{0}^{2}-\omega_2^{2}\right)}{2}v_2 - \frac{\gamma \omega_2}{2}u_2 +\frac{3}{8} \alpha v_2(u_2^2 + v_2^2) +\frac{3}{4} \alpha v_2(u_1^2+v_1^2)& = 0.
  \end{aligned}
\end{equation*}
In \cite[Section 6.2]{BBPMZ23}, the authors relax the system to the following parametrized polynomial system with variables $u_i,v_i$ and parameters $a_{i,j}$:
\begin{equation}
  \label{eq:oscillatorsRelaxed}
  \begin{aligned}
    \dcf_1&\coloneqq a_{1,0}+a_{1,1}u_1 + a_{1,2}v_1 + a_{1,3} u_1(u_1^2 + v_1^2) + a_{1,4} u_1(u_2^2+v_2^2), \\
    \dcf_2&\coloneqq a_{2,0}+a_{2,1}u_1 + a_{2,2}v_1 + a_{2,3} v_1(u_1^2 + v_1^2) + a_{2,4} v_1(u_2^2+v_2^2),\\
    \dcf_3&\coloneqq a_{3,0}+a_{3,1}u_2 + a_{3,2}v_2 + a_{3,3} u_2(u_1^2+v_1^2) + a_{3,4} u_2(u_2^2 + v_2^2),\\
    \dcf_4&\coloneqq a_{4,0}+a_{4,1}u_2 + a_{4,2}v_2 + a_{3,3} v_2(u_1^2+v_1^2) + a_{4,4} v_2(u_2^2 + v_2^2),
  \end{aligned}
\end{equation}
and use their theory to deduce an upper bound of $25$ solutions for $\CC^4$.

One way to apply our techniques to System~\eqref{eq:oscillatorsRelaxed} is to note that it is a horizontal system with tropically transverse base.  By omitting some of the new variables $y,z$, its two-stage modification can be simplified to:
\begin{align}
  \label{eq:oscillatorsRelaxedModified}
  \nonumber \hat \dcf_1&\coloneqq a_{1,0}+a_{1,1}u_1 + a_{1,2}v_1 + a_{1,3} z_{1,1} + a_{1,4} z_{1,2} & \hat \dcg_1 &\coloneqq z_{1,1} - u_1y_1 & \hat \dch_1 &\coloneqq z_{1,2} - u_1y_2 \\
  \nonumber \hat \dcf_2&\coloneqq a_{2,0}+a_{2,1}u_1 + a_{2,2}v_1 + a_{2,3} z_{2,1} + a_{2,4} z_{2,2} & \hat \dcg_2 &\coloneqq z_{2,1} - v_1y_1 & \hat \dch_2 &\coloneqq z_{2,2} - v_1y_2 \\
  \nonumber \hat \dcf_3&\coloneqq a_{3,0}+a_{3,1}u_2 + a_{3,2}v_2 + a_{3,3} z_{3,1} + a_{3,4} z_{3,2} & \hat \dcg_3 &\coloneqq z_{3,1} - u_2y_1 & \hat \dch_3 &\coloneqq z_{3,2} - u_2y_2 \\
  \nonumber \hat \dcf_4&\coloneqq a_{4,0}+a_{4,1}u_2 + a_{4,2}v_2 + a_{4,3} z_{4,1} + a_{4,4} z_{4,2} & \hat \dcg_4 &\coloneqq z_{4,1} - v_2y_1 & \hat \dch_4 &\coloneqq z_{4,2} - v_2y_2 \\
  \hat \dck_1 &\coloneqq y_1 - (u_1^2+v_1^2)\qquad \hat\dck_2 \coloneqq y_2 - (u_2^2+v_2^2)
\end{align}
By \cref{thm:transverseHorizontalModification}, the generic root count is the mixed volume of the system above, which reconfirms the upper bound of $25$ in \cite{BBPMZ23}. Explicit homotopies constructed from this modification can be found in the GitHub repository.

We note that while the techniques in \cite{BBPMZ23} require the polynomial support to have a constant term (see \cite[Theorem 3.8]{BBPMZ23}), our techniques do not.  Hence, if one is only interested in the solutions inside the torus $(\CC^\ast)^4$, one can set $a_{i,0}=0$ in Systems~\eqref{eq:oscillatorsRelaxed} and \eqref{eq:oscillatorsRelaxedModified} to obtain a smaller upper bound of $16$.

\subsection{Equilibria of the Kuramoto model}\label{sec:Kuramoto}
In this section, we consider the Kuramoto equations with phase delays as discussed by Chen, Korchevskaia, and Lindberg in \cite[Section 2.4]{ChenKorchevskaiaLindberg2022}. For a graph $G$ with vertices $[n]$, these form a parametrized system with variables $x$ and parameters $\overline w$, $a$, $C$:
\begin{equation}
  \label{eq:kuramoto}
  f_{G,i} = \overline w_i - \sum_{j\in N_G(i)} a_{ij} \Big(\frac{x_iC_{ij}}{x_j}-\frac{x_j}{x_iC_{ij}}\Big) \quad\text{for } i\in [n],
\end{equation}
where  $N_G(i)$ denotes the set of vertices adjacent to the vertex $i$.

In \cite[Corollary 3.9]{ChenKorchevskaiaLindberg2022}, the authors show that the number of solutions equals the (normalized) volume of the adjacency polytope of $G$.  This result is beyond a simple application of our techniques. However, we can easily express the generic root count as a mixed volume, by viewing System~\eqref{eq:kuramoto} as a horizontal system with relaxed polynomial support as discussed in \cref{sec:horziontalFamilyRelaxation}.  Consider its modification:
\begin{equation}
  \label{eq:kuramotoModified}
  \begin{aligned}
    \hat \dcf_{G,i} &= \overline w_i - \sum_{j\in N_G(i)} a_{ij} y_{ij} &&\text{for } i\in [n],\\
    \hat \dcg_{ij} &= y_{ij} - \Big(\frac{x_iC_{ij}}{x_j}-\frac{x_j}{x_iC_{ij}}\Big) &&\text{for } \{ij\}\in E(G),
  \end{aligned}
\end{equation}
and observe that all resulting tropical hypersurfaces intersect transversally.  Hence, the generic root count of System~\eqref{eq:kuramoto} is the mixed volume of System~\eqref{eq:kuramotoModified}.

\subsection{Realizations of Laman graphs}\label{sec:graphRigidity}
In \cite{cggkls}, the authors study the problem of finding realizations of minimally 2-rigid so-called \emph{Laman graphs}, given generic assignments of the edge lengths. Such realizations correspond to solutions of a parametrized system consisting of equations that are either linear or linear in the reciprocals of the variables, which can readily be treated with our techniques. 

As an example, consider the Laman graph $G$ in \cref{fig:Laman}(a), for which the realizations correspond to the solutions of the system
\begin{equation}
  \label{eq:graphRigidity}
  \begin{aligned}
    \dcf_{1}&\coloneqq x_{12}+x_{23}-x_{13}, & \dcg_{1}&\coloneqq \lambda_{12}x_{12}^{-1}+\lambda_{23}x_{23}^{-1}-\lambda_{13}x_{13}^{-1}, & \dch&\coloneqq x_{23}-1,\\
    \dcf_{2}&\coloneqq x_{23}+x_{34}-x_{24}, & \dcg_{2}&\coloneqq \lambda_{23}x_{23}^{-1}+\lambda_{34}x_{34}^{-1}-\lambda_{24}x_{24}^{-1}. \\
  \end{aligned}
\end{equation}
with variables $x_{ij}$ encoding the edge directions, and parameters $\lambda_{ij}$  encoding the edge lengths.  Note that this  is a slightly simplified version of the system in \cite[Example 2.25]{cggkls}, where we have substituted $y_{ij}$ by $\lambda_{ij}x_{ij}^{-1}$, and $x_{ji}$ by $-x_{ij}$ for $i<j$.

Consider the parametrized ideals $\dcI\coloneqq\langle \dcf_{1},\dcf_{2},\dcg_{1},\dcg_{2},h_Q \rangle$, $\dcI_f\coloneqq \langle \dcf_{1},\dcf_{2} \rangle$, and $\dcI_g\coloneqq \langle \dcg_{1},\dcg_{2} \rangle$, and let $P=(1,1,1,1,1)$ be the target parameters. To compute the tropical data required for constructing the homotopies, note that for $Q=t^v\cdot P$ for generic $v\in\mathbb{Z}^{E(G)}$, it holds that
\[ \Trop(\dcI_Q) = \Trop(\dcI_{f,Q}) \wedge \Trop(\dcI_{g,Q}) \wedge \Trop(\dch_Q), \]
since the intersection is transversal, as sketched in \cref{fig:Laman}(b).  For $v_{ij}=i+j$, we obtain the following four intersection points, each of multiplicity $1$:
\begin{align*}
  (0, 1, 0, 1, 0),\;
  (0, 1, 0, 0, 2),\;
  (-2, -2, 0, 1, 0),\;
  (-2, -2, 0, 0, 2) \in \RR^{\{12,13,23,24,34\}}.
\end{align*}
This shows that the generic root count is equal to $4$, and the four intersection points can be used to compute the solutions of System~\eqref{eq:graphRigidity} for the parameters $P$ through \cref{alg:tropicalHomotopies}. For example, the intersection point $w=e_{13}+e_{24}$ yields the homotopy
\[H^{(w)}=\Big(x_{12} + x_{23} - t x_{13}, x_{23} + x_{34} - t x_{24}, t^2x_{23}^{-1} + x_{12}^{-1} - x_{13}^{-1}, t^2x_{34}^{-1} + x_{23}^{-1} - x_{24}^{-1}, -1 + x_{23}\Big),\]
and the binomial start system
\[\Big(x_{12} + x_{23}, x_{23} + x_{34}, x_{12}^{-1} - x_{13}^{-1}, x_{23}^{-1} - x_{24}^{-1}, -1 + x_{23}\Big),\]
which has a unique solution in the torus.  Homotopies for the other intersection points can be found in our GitHub repository.

\begin{figure}[t]
  \centering
  \begin{tikzpicture}
    \node (left)
    {
      \begin{tikzpicture}
        \node[draw,circle] (v1) at (0,0) {$1$};
        \node[draw,circle] (v2) at (-1,-1.5) {$2$};
        \node[draw,circle] (v3) at (1,-1.5) {$3$};
        \node[draw,circle] (v4) at (0,-3) {$4$};
        \draw
        (v1) -- (v2)
        (v1) -- (v3)
        (v2) -- (v3)
        (v2) -- (v4)
        (v3) -- (v4);
      \end{tikzpicture}
    };
    \node[anchor=north] at (left.south) {(a)};
    \node[anchor=west,xshift=5mm] (right) at (left.east)
    {
      \begin{tikzpicture}[x={(240:7mm)},y={(0:10mm)},z={(90:5mm)},every node/.style={font=\footnotesize}]
        \fill[blue!20]
        (-3,-3,0) -- (2,-3,0) -- (2,4,0) -- (-3,4,0) -- cycle;
        \node[anchor=north west,darkblue] at (2,-3,0) {$\Trop(\dch_Q)$};
        \draw[very thick,darkgreen]
        (0,0,0) -- ++(-3,-3,0)
        (0,0,0) -- ++(2,0,0)
        (0,0,0) -- node[above,pos=0.75] {$\Trop(\dcI_{f,Q})$} ++(0,4,0);
        \node[anchor=north west,darkgreen] at (0,0,0) {$(0,0,0)$};
        \fill[darkgreen] (0,0,0) circle (2pt);

        \draw[very thick,darkred]
        (-1,1,0) -- ++(3,3,0)
        (-1,1,0) -- ++(-2,0,0)
        (-1,1,0) -- node[below,pos=0.75] {$\Trop(\dcI_{g,Q})$} ++(0,-4,0);
        \node[anchor=south east,darkred] at (-1,1,0) {$(-1,1,0)$};
        \fill[darkred] (-1,1,0) circle (2pt);
        \fill[white,draw=black]
        (0,2,0) circle (2pt)
        (-1,-1,0) circle (2pt);
      \end{tikzpicture}
    };
    \node[anchor=north] at (right.south) {(b)};
  \end{tikzpicture}
  \caption{A Laman graph and its resulting tropical intersection}
  \label{fig:Laman}
\end{figure}
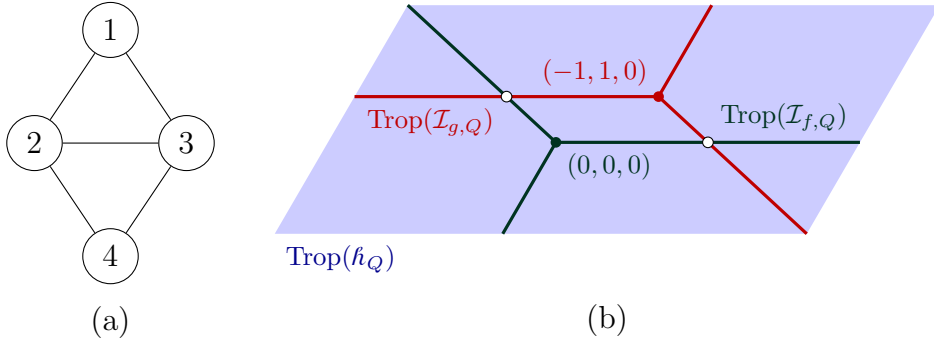

%%% Local Variables:
%%% mode: LaTeX
%%% TeX-master: "main"
%%% ispell-local-dictionary: "en_US"
%%% End:

%% file: conclusion.tex
\section{Conclusion and outlook}\label{sec:conclusion}
In this work, we have explained how to construct generically optimal homotopies for para\-me\-trized polynomial systems from their tropical data, and listed conditions on the systems that result in desirable properties of the resulting homotopies (\cref{sec:tropicalHomotopies}).
We then discussed how said tropical data can be obtained for two classes of parametrized polynomial systems:  vertically parametrized systems (\cref{sec:verticalFamilies}) and horizontally parametrized systems (\cref{sec:horizontalFamilies}). 

The data of vertically parametrized systems can be obtained from the intersection of a tropical linear space and a tropical binomial variety.  In contrast, the data of horizontally parametrized systems is more difficult to obtain, and we have proposed two relaxations.
Finally, we have highlighted several examples from the literature where our techniques can be applied (\cref{sec:caseStudies}).

{\samepage
Two main challenges for the efficient application of our techniques remain:
\begin{enumerate}
\item develop and implement tropical homotopy continuation to speed up the computation of the required tropical data (see timings in \cref{sec:WNT}; ongoing work, see \cite{DaiseyRen2024} for a preliminary paper),
\item implement a numerically robust path tracker for tropical homotopies.
\end{enumerate}
}

Another future research direction is to identify more classes of parametrized systems whose generic root count is below their mixed volume and whose tropical data is efficiently computable.  Such classes are important targets for relaxation, independent of whether they arise directly in practise.
%%% Local Variables:
%%% mode: LaTeX
%%% TeX-master: "main"
%%% ispell-local-dictionary: "en_US"
%%% End: